\documentclass{article}
\usepackage[top=4cm, bottom=4cm, left=3cm, right=3cm]{geometry}
\usepackage{graphics}
 \usepackage{graphicx}
 \usepackage{epsfig}
 \usepackage{amsmath}
\usepackage{amsfonts}
\usepackage{amssymb}
\usepackage{xcolor}
\usepackage{enumerate}
\usepackage{hyperref}
\hypersetup{colorlinks=true,linkcolor=blue,filecolor=mangeta}
\newtheorem{theorem}{Theorem}[section]

\newtheorem{proposition}{Proposition}[section]
\newtheorem{lemma}{Lemma}[section]

\newtheorem{remark}{Remark}

\newenvironment{proof}[1][Proof]{\textbf{#1.} }{\ \rule{0.5em}{0.5em}}

\numberwithin{equation}{section}
\def \dis {\displaystyle}

\def \R {\mathbb{R}}

\def \O {\mathcal{O}}

\def\dT{dx\, dt}

\def \hvarphi \widehat{\varphi}

\def\dq{dx\,dt}

\def \hvarphi \widehat{\varphi}

\begin{document}
  \title{Stackelberg exact controllability of a class of nonlocal parabolic equations}
 \author{{{Landry Djomegne} \thanks{{\it
 				University of Dschang, BP 67 Dschang, Cameroon, West region,
 				email~: {\sf landry.djomegne\char64yahoo.fr} }}}
 	{{ \quad Cyrille Kenne} \thanks{{\it Laboratoire LAMIA, Universit\'e des Antilles, Campus Fouillole, 97159 Pointe-{\`a}-Pitre Guadeloupe (FWI)-Laboratoire L3MA, UFR STE et IUT, Universit\'e des Antilles, Schoelcher, Martinique,
 				email~: {\sf kenne853\char64gmail.com} }}}
 }

  \date{\today}
  
  \maketitle	
  	

\begin{abstract}
This paper deals with a multi-objective control problem for a class of nonlocal parabolic equations, where the non-locality is expressed through an integral kernel. We present the Stackelberg strategy that combines the concepts of controllability to trajectories with optimal control. The strategy involves two controls: a main control (the leader) and a secondary control (the follower). The leader solves a controllability to trajectories problem which consists to drive the state of the system to a prescribed target at a final time while the follower solves an optimal control problem which consists to minimize a given cost functional. The paper considers two cases: in the first case, both the leader and the follower act in the interior of the domain, and in the second case, the leader acts in the interior of the domain and the follower acts on a small part of the boundary. These results are applied to both linear and nonlinear nonlocal parabolic systems.
\end{abstract}

\textbf {2010} Mathematics Subject Classification. {35K05; 49J20; 93B05; 93B07; 93C20.}\par

\section{Introduction}\paragraph{}

Let $\Omega\subset\R^N,\ N\geq 1$, be a bounded open set with a regular boundary $\partial\Omega$. Let $\omega$ and $\O$ be two arbitrary nonempty open sets of $\Omega$ such that $\omega\cap\O=\emptyset$ and let $\Gamma$ be a nonempty open subset of $\partial\Omega$.  For the  real number $T > 0$, we denote $Q =(0, T )\times \Omega$ and $\Sigma =(0, T )\times \partial\Omega$. Let $K(t,x,\theta)\in L^\infty(Q\times \Omega)$ represents an integral kernel.

 In the sequel, we will denote the norm in $L^\infty(Q\times \Omega)$ by $\|\cdot\|_{\infty}$ and the symbol $C$ is used to design a generic positive constant whose value can change even from line to line. We also denote by $\nu=\nu(x)$ the outward unit  normal vector at the point $x\in \partial\Omega$.
  We consider the following linear nonlocal parabolic systems:
  \begin{equation}\label{eq}
  \left\{
  \begin{array}{rllll}
  \dis y_{t}-\Delta y+\int_\Omega K(t,x,\theta)y(t,\theta)\ d\theta &=&f\chi_{\omega}+v\chi_{\O}& \mbox{in}& Q,\\
  \dis y&=&0& \mbox{on}& \Sigma, \\
  \dis  y(0,\cdot)&=&y^0 &\mbox{in}&\Omega
  \end{array}
  \right.
  \end{equation}
and
\begin{equation}\label{eqbis}
	\left\{
	\begin{array}{rllll}
		\dis q_{t}-\Delta q+\int_\Omega K(t,x,\theta)q(t,\theta)\ d\theta &=&g\chi_{\omega}& \mbox{in}& Q,\\
		\dis q&=&u\textbf{1}_{\Gamma}& \mbox{on}& \Sigma, \\
		\dis  q(0,\cdot)&=&q^0 &\mbox{in}&\Omega,
	\end{array}
	\right.
\end{equation}
where $y=y(t,x)$ and $q=q(t,x)$ denote the states, the functions $f,\ g,\  u$ and $v$ are the controls given in appropriate spaces, $y^0$ and $q^0$ are given initial datum. Here, $\chi_{\mathcal{A}}$ denotes the characteristic function of the set $\mathcal{A}$ and $\textbf{1}_{\Gamma}\in \mathcal{C}^2(\partial\Omega)$ is a smooth nonnegative function such that supp $\textbf{1}_{\Gamma}=\overline{\Gamma}$. We denote by $y_{t}$ the partial derivative of $y$ with respect to $t$.
  
In this work, we develop the Stackelberg strategy \cite{Lions1994Stackelberg1} for the nonlocal heat equations. We assume that we can act on the systems  through a hierarchic of controls. More precisely, we combine the optimal control strategy with the exact controllability to the trajectories of systems \eqref{eq} and \eqref{eqbis}.

The concept of Stackelberg competition introduced in 1934 \cite{Von1934Stackelberg} is a game-theoretic approach where one player (the leader) makes a strategic decision first, and the second player (the follower) then respond based on the leader's decision.
 
 In the framework of PDEs, the hierarchic control was introduced by J-L. Lions in 1994 \cite{Lions1994Stackelberg1, Lions1994Stackelberg2} to study a bi-objective control problem for the wave and heat equation respectively. In theses works, the author acted on the different systems with two controls. The leader solving an approximate controllability problem while the follower solving an optimal control problem. In the past years, many other researchers have used hierarchic control in the sense of Lions (see for instance \cite{djomegnebackward, djomegne2023, kere2017coupled, mercan1, mercan2, mercan3, Nakoulima2007, romario2018}). Recently, the authors in \cite{djomegnelinear,teresa2018} used the hierarchic control combining the concept of controllability  with robustness.\par 
 
All the previous work deal with hierarchic strategy with distributed controls. The first paper which solve with the boundary controls is the one given by \cite{araruna2019boundary}. In that paper, the authors dealt with the Stackelberg-Nash exact controllability of parabolic equations with the possibility of the leader and the followers being placed on the boundary. In \cite{liliana2020}, the authors extend and discuss the results concerning the robust hierarchic strategy for the heat equation using boundary controls. Recently, the authors in \cite{carreno2023null} considered a multi-objective control problem for the Kuramoto-Sivashinsky equation with a distributed control called leader and two boundary controls called followers.

In all the above cited works, the hierarchic strategy were applied to local systems. In this paper, we extend the results concerning the hierarchic strategy to nonlocal parabolic equations. This extension introduces additional difficulties, particularly when establishing new observability inequalities of Carleman for adjoint systems associated with equations \eqref{eq} and \eqref{eqbis}. To the best of our knowledge, this has not been done before.
  
In this paper, we extend the Stackelberg control of Lions to nonlocal systems with an integral kernel \eqref{eq} and \eqref{eqbis}. To fix ideas, we begin by explaining the strategy control  to system \eqref{eq}. To be more specific, we introduce a nonempty open set $\O_d\subset \Omega$, representing an observation domain of the follower and define the cost functional
  \begin{equation}\label{all16}
  J(f;v)=\frac{1}{2}\int_0^T\int_{\O_d}|y-y_d|^2\ dxdt+\frac{\mu}{2}\int_0^T\int_{\O}|v|^2\ dxdt,
  \end{equation}
  where $\mu$ is a positive constant and $y_d\in L^2((0,T)\times \O_d)$ is given. The structure of the control process we will follow is described as follows:
\begin{enumerate}
	\item Once the leader $f$ has been fixed, we look for a control $\hat{v}$ depending of $f$ which solves the following problem
	\begin{equation}\label{Nash1}
		\begin{array}{rllll}
			\dis	J(f;\hat{v})=\inf_{v}J(f;v).
		\end{array}
	\end{equation}
The function $\hat{v}$ satisfied \eqref{Nash1} is called an optimal control for $J$.
Since the functional $J$ is convex (because the system \eqref{eq} is linear), then  $\hat{v}$ is an optimal control for $J$ if and only if
\begin{equation}\label{Nashprime}
	\frac{\partial J}{\partial w}(f;\hat{v})\cdot w=0,\ \ \forall w\in L^2((0,T)\times \O).
\end{equation}
	
	\item Let us consider an uncontrolled trajectory of \eqref{eq}, that is a function $\bar{y}=\bar{y}(t,x)$ solution of
	\begin{equation}\label{eqbar}
		\left\{
		\begin{array}{rllll}
			\dis \bar{y}_{t}-\Delta \bar{y}+\int_\Omega K(t,x,\theta)\bar{y}(t,\theta)\ d\theta &=&0& \mbox{in}& Q,\\
			\dis \bar{y}&=&0& \mbox{on}& \Sigma, \\
			\dis  \bar{y}(0,\cdot)&=&\bar{y}^0 &\mbox{in}&\Omega.
		\end{array}
		\right.
	\end{equation}
	Once the optimal control has been identified and fixed for each $f$, we look for a control $\hat{f}$ that satisfies
	\begin{equation}\label{mainobj}
		y(T,x)=	\bar{y}(T,x)=0\ \mbox{in}\ \Omega.
	\end{equation}	
\end{enumerate} 

Now, we are also interested to study the system \eqref{eqbis}. The same methodology presented above can be used here. The cost functional \eqref{all16} should be replaced by 

 \begin{equation}\label{j}
	\widetilde{J}(g;u)=\frac{1}{2}\int_0^T\int_{\O_d}|q-q_d|^2\ dxdt+\frac{\mu}{2}\int_0^T\int_{\Gamma}|u|^2\ d\sigma dt,
\end{equation}
 where again $\mu$ is a positive constant and $q_d\in L^2((0,T)\times \O_d)$ is a given functional. The control process can be described as follows.
\begin{enumerate}
	\item For each leader $g$, we find a control $\hat{u}$ which solves the following problem
	\begin{equation}\label{op}
		\begin{array}{rllll}
			\dis	\widetilde{J}(g;\hat{u})=\inf_{u}\widetilde{J}(g;u).
		\end{array}
	\end{equation}
The function $\hat{u}$ satisfied \eqref{op} is called an optimal control for $\widetilde{J}$ given by  \eqref{j}.
Since the functional $\widetilde{J}$ is convex, then  $\hat{u}$ is an optimal control for $\widetilde{J}$ if and only if
\begin{equation}\label{hatu}
		\frac{\partial \widetilde{J}}{\partial w}(g;\hat{u})\cdot w=0,\ \ \forall w\in L^2((0,T)\times \Gamma).
\end{equation}

	\item Let us  fix an uncontrolled trajectory of system \eqref{eqbis}, that is a function $\bar{q}=\bar{q}(t,x)$ solution to
	\begin{equation}\label{eqbisbar}
		\left\{
		\begin{array}{rllll}
			\dis \bar{q}_{t}-\Delta \bar{q}+\int_\Omega K(t,x,\theta)\bar{q}(t,\theta)\ d\theta &=&0& \mbox{in}& Q,\\
			\dis \bar{y}&=&0& \mbox{on}& \Sigma, \\
			\dis  \bar{q}(0,\cdot)&=&\bar{q}^0 &\mbox{in}&\Omega.
		\end{array}
		\right.
	\end{equation}
	We look for a control $\hat{g}$ verifying
	\begin{equation}\label{obj}
		q(T,x)=	\bar{q}(T,x)=0\ \mbox{in}\ \Omega.
	\end{equation}	
\end{enumerate} 

As mentioned in \cite{lu2016null}, the study of the controllability of systems \eqref{eq} and \eqref{eqbis} is motivated by many relevant applications from physics, chemotaxis, biology and ecology; we can see \cite{kavallaris2018} for instance where many models are introduced. Such systems like \eqref{eq} and \eqref{eqbis} can appear for instance in population dynamics where the state $y(t, x)$ represents the density of the species at time $t$ and position $x$, while the nonlocal term $\int_\Omega K(t,x,\theta)y(t,\theta)\ d\theta$ is considered as the rate of reproduction. This integral term is a way to express that the evolution of the species in a point of space depends on the total
amount of the species \cite{genninonlocal}.

Some results are available in the literature concerning the controllability of parabolic nonlocal systems. The first result concerning this subject is the one studied in \cite{deleo2014}, where the authors established the exact controllability of the Schr\"{o}dinger equation with a nonlocal term. In \cite{lu2016null}, the authors proved
an interior null controllability result for the system like \eqref{eq} when the kernel is time-independent and analytic. They used some compactness-uniqueness arguments in order to obtain the unique continuation properties. A similar result for the wave equation has also been derived. The results of \cite{lu2016null} have been extended in \cite{lissy2018internal} to a general coupled parabolic system. Moreover, similar results have been obtained in \cite{micu2018local} for a one-dimensional equation with a time-independent kernel in separated variables, by means of spectral analysis techniques. 
Later on, the authors in \cite{biccari2019null} have extended these mentioned results, by relaxing the assumptions on the kernel. The authors studied both the linear and the semilinear case, by using a Carleman approach. Recently, in \cite{genninonlocal}, the authors have proved the null controllability property for the two degenerate nonlocal systems using the Kakutani’s fixed point Theorem. The authors in \cite{balch2021null} have proved the internal null-controllability of a heat equation with Dirichlet boundary conditions and perturbed by a semilinear nonlocal term.

\subsection{Main results}
In this paper, we use the Carleman approach to solve the null controllability of the nonlocal systems \eqref{eq} and \eqref{eqbis}. Using the same approach as the authors of \cite{biccari2019null}, we assume that the integral kernel $K\in L^\infty(Q\times \Omega)$ satisfies the following assumption:
 \begin{equation}\label{k}
 	\mathcal{K}=:\sup_{(t,x)\in \overline{Q}}\exp\left(\frac{\sigma^-}{l^4(t)}\right)\int_\Omega|K(t,x,\theta)|\ d\theta<+\infty,
 \end{equation}
 where  the functions $\sigma^-$ and $l(t)$ are defined by \eqref{sigmamoins} and \eqref{l}, respectively. The assumption \eqref{k} means that, the kernel $K=K(t,x,\theta)$ is bounded and decrease exponentially when $t$ goes to $0^+$ and to $T^-$ \cite{biccari2019null}.
 
We firstly address the Stackelberg control of system \eqref{eq} where all the controls are localized on the interior of the domain. We have the following result.

\begin{theorem}\label{theolinear}
	
  Assume that $\dis \O_d\cap\omega\neq\emptyset$ and the kernel $K$ satisfies \eqref{k}. Let $\bar{y}$ be the unique solution to \eqref{eqbar} associated to the initial state $\bar{y}^0$. Suppose that $y^0\in L^2(\Omega)$ and $\mu$ is large enough. Then,
	there exists a positive real weight function $\varpi_1=\varpi_1(t)$ blowing up at $t=T$ such that for any $y_{d}\in L^2((0,T)\times\O_d)$ satisfying
	\begin{equation}\label{nou}
		\int_{0}^{T}\int_{\O_d}\varpi_1^{2}|\bar{y}-y_{d}|^2\ dxdt<+\infty,
	\end{equation}
	 there exist a control $\hat{f}\in L^2((0,T)\times \omega)$ and a unique optimal control $\hat{v}\in L^2((0,T)\times \O)$ such that the corresponding solution to \eqref{eq} satisfies \eqref{mainobj}. 
\end{theorem}
 
Now, we are interested to the system \eqref{eqbis} where the leader control is applied on the interior of the domain and the follower is localized on the boundary. In the following, we need the Hilbert spaces 
\begin{equation}\label{hrsa}
	H^{r,s}(Q)=L^2((0,T); H^r(\Omega))\cap H^s((0,T); L^2(\Omega)),
\end{equation}
for $r, s\in \R_+$ endowed with norms
\begin{equation}\label{hrs}
	\|z\|_{H^{r,s}(Q)}=\left(\|z\|^2_{L^2((0,T);H^r(\Omega))}+\|z\|^2_{H^s((0,T); L^2(\Omega))}\right)^{1/2}.
\end{equation}

\begin{remark}
	We can replace in \eqref{hrs} the domain $Q$ and $\Omega$ by $\Sigma$ and $\partial\Omega$, respectively. Then, we obtain the analogous space for functions defined on the boundary.
\end{remark}

We have the following result.

\begin{theorem}\label{theoboundary}
	Assume that $\dis \O_d\cap\omega\neq\emptyset$ and the kernel $K$ satisfies \eqref{k}. Assume that $q^0\in L^2(\Omega)$ and $\mu$ is sufficiently large. Let $\bar{q}$ be the unique solution to \eqref{eqbisbar} associated to the initial state $\bar{q}^0$. Then, there exists a positive real weight function $\varpi_2=\varpi_2(t)$ blowing up at $t=T$ such that for any $q_{d}\in L^2((0,T)\times\O_d)$ verifying
	\begin{equation}\label{noubis}
		\int_{0}^{T}\int_{\O_d}\varpi_2^{2}|\bar{q}-q_{d}|^2\ dxdt<+\infty,
	\end{equation}
	we can find a leader control $\hat{g}\in L^2((0,T)\times \omega)$ and a unique optimal control $\hat{v}\in H^{1/2,1/4}((0,T)\times \Gamma)$ such that the corresponding solution to \eqref{eqbis} satisfies \eqref{obj}. 
\end{theorem}

 The rest of the paper is organized as follows. In Section \ref{interiorpart}, we analyse the Stackelberg strategy concerning the case of distributed leader and follower, that is, the case of system \eqref{eq}. In this part, we firstly prove the existence and characterization of optimal control, then we prove some suitable Carleman inequalities and deduce the exact controllability result. Finally, we extend the results to the semilinear case. The Section \ref{boundary} concerns the control of system \eqref{eqbis} with distributed leader and boundary follower. In this section, we follow similar arguments as in Section \ref{interiorpart}.

 \section{Exact controllability with distributed leader and follower}\label{interiorpart}

The first step of this part is to reduce the problem of exact controllability to trajectory to a null controllability problem. So,
let us consider $z:=y-\bar{y}$, where $y$ and $\bar{y}$ are respectively solutions to \eqref{eq} and \eqref{eqbar}. The property \eqref{mainobj} is equivalent to a null controllability property for $z$, that is
\begin{equation}\label{znull}
	z(T,x)=0\ \mbox{in}\ \Omega,
\end{equation}	
where $z$ is solution of the system
\begin{equation}\label{eqz}
	\left\{
	\begin{array}{rllll}
		\dis z_{t}-\Delta z+\int_\Omega K(t,x,\theta)z(t,\theta)\ d\theta &=&f\chi_{\omega}+v\chi_{\O}& \mbox{in}& Q,\\
		\dis z&=&0& \mbox{on}& \Sigma, \\
		\dis  z(0,\cdot)&=&z^0 &\mbox{in}&\Omega,
	\end{array}
	\right.
\end{equation}
with $z^0:=y^0-\bar{y}^0$.  From now on, we are going to work with system \eqref{eqz} instead of \eqref{eq}.

The functional $J$ given by \eqref{all16} can be rewritten as follows
  \begin{equation}\label{jz}
	J(f;v)=\frac{1}{2}\int_0^T\int_{\O_d}|z-z_d|^2\ dxdt+\frac{\mu}{2}\int_0^T\int_{\O}|v|^2\ dxdt,
\end{equation}
where $z_d:=y_d-\bar{y}$.

Let us introduce the Hilbert space 
\begin{equation}\label{w}
	W(Q)= \left\{z \in L^2((0,T);H^1_0(\Omega)); z_t\in L^2\left((0,T);H^{-1}(\Omega)\right)\right\},
\end{equation}
equipped with the norm given by
\begin{equation}\label{}
	\|z\|^2_{	W(Q)}=\|z\|^2_{L^2((0,T);H^1_0(\Omega))}+\|z_t\|^2_{L^2\left((0,T);H^{-1}(\Omega)\right)}.
\end{equation}
It is well known (see, for instance \cite[Page 37]{lions2013}) that for any $y^0\in L^2(\Omega)$, $f\in L^2((0,T)\times \omega)$ and $v\in L^2((0,T)\times \O)$, the system \eqref{eqz} admits a unique weak solution $z\in W(Q)$. Moreover there exists a constant $C=C(\|K\|_{\infty}, T)>0$ such that the following estimation holds:
\begin{equation}\label{}
	\begin{array}{llllll}
		\dis \|z(T,\cdot)\|^2_{L^2(\Omega)}+\|z\|^2_{W(Q)}\leq 
		C\left(\|f\|^2_{L^2((0,T)\times \omega)}+\|v\|^2_{L^2((0,T)\times \O)}+\|z^0\|^2_{L^2(\Omega) }\right).
	\end{array}
\end{equation}

\subsection{Existence, uniqueness and characterization of the optimal control}\label{low}

In this section, we are interested in the classical optimal control problem \eqref{Nash1}. To this end, we first prove the existence and the uniqueness of a minimizer and secondly, we characterize it through a suitable optimality system.

Due to the fact that the functional $J$ given by \eqref{jz} is quadratic and the system \eqref{eqz} is linear, it is not difficult to prove that $J$ is continuous and strictly convex. Moreover, we have 
$$
J(f;v)\geq \frac{\mu}{2}\|v\|^2_{L^2((0,T)\times \O)}.
$$
The above inequality allow us to say that $J$ is coercive. These three properties guarantee the existence of a unique solution $\hat{v}\in L^2((0,T)\times \O)$ to the optimal control problem \eqref{Nash1}-\eqref{jz}.

In order to characterize the optimal control $\hat{v}$, we use the Euler-Lagrange optimality condition \eqref{Nashprime}, i.e.
	\begin{equation}\label{euler}
	\lim\limits_{\lambda\rightarrow 0}\frac{J(f;\hat{v}+\lambda w)-J(f;\hat{v})}{\lambda}=0,\ \ \mbox{for all}\ \ w\in L^2((0,T)\times\O).
\end{equation}
After some calculations, \eqref{euler} gives
\begin{equation}\label{EL}
	\begin{array}{rll}
		&&\dis 	\int_0^T\int_{\O_d}(z-z_d)h\ \dT
		+\mu\int_0^T\int_{\O}\hat{v}w\ \dT=0,\ \mbox{for all}\  w\in L^2((0,T)\times \O),
	\end{array}
\end{equation}
where $h$ is the solution to 
\begin{equation}\label{zop}
	\left\{
	\begin{array}{rllll}
		\dis h_{t}-\Delta h+\int_\Omega K(t,x,\theta)h(t,\theta)\ d\theta &=&w\chi_{\O}& \mbox{in}& Q,\\
		\dis h&=&0& \mbox{on}& \Sigma, \\
		\dis  h(0,\cdot)&=&0 &\mbox{in}&\Omega.
	\end{array}
	\right.
\end{equation}

Now, we introduce the adjoint state to system \eqref{zop}:
\begin{equation*}
	\left\{
	\begin{array}{rllll}
		\dis -p_{t}-\Delta p+\int_\Omega K(t,x,\theta)p(t,\theta)\ d\theta &=&(z-z_d)\chi_{\O_d}& \mbox{in}& Q,\\
		\dis p&=&0& \mbox{on}& \Sigma, \\
		\dis  p(T,\cdot)&=&0 &\mbox{in}&\Omega.
	\end{array}
	\right.
\end{equation*}
If we multiply the first equation of this latter system by $h$ solution of \eqref{zop} and we integrate by parts over $Q$, one obtain using the Fubini's theorem
\begin{equation*}
	\begin{array}{rll}
		\dis \int_0^T\int_{\O_d}(z-z_d)h\ \dT =\int_0^T\int_{\O}wp\ \dT.
	\end{array}
\end{equation*}
Combining this above equality with \eqref{EL}, we obtain
\begin{equation*}
	\begin{array}{rll}
		\dis \int_0^T\int_{\O}w \left(p+\mu\hat{v}\right)\ \dT=0\,\, \mbox{for all}\  w\in L^2((0,T)\times \O),
	\end{array}
\end{equation*}
from which we deduce $p+\mu\hat{v}=0\ \ \mbox{in}\ \ (0,T)\times\O$.

We have proved the following result giving the optimality system that characterizes the optimal control for the cost functional $J$.

\begin{proposition}\label{quasi}
	
Let $z^0\in L^2(\Omega)$ and $f\in L^2((0,T)\times \omega)$ be given. Then, the optimal control problem \eqref{Nash1}-\eqref{jz} has a unique solution
 $\hat{v}\in L^2((0,T)\times \O)$.  Furthermore, 
\begin{equation}\label{vop}
	\hat{v}=-\frac{1}{\mu}p\ \ \mbox{in}\ \ \ (0,T)\times \O,	
\end{equation}
where  $(z,p)$ is the solution to the following optimality systems
\begin{equation}\label{yop}
	\left\{
	\begin{array}{rllll}
		\dis z_{t}-\Delta z+\int_\Omega K(t,x,\theta)z(t,\theta)\ d\theta &=&\dis f\chi_{\omega}-\frac{1}{\mu}p\chi_{\O}& \mbox{in}& Q,\\
		\dis z&=&0& \mbox{on}& \Sigma, \\
		\dis  z(0,\cdot)&=&z^0 &\mbox{in}&\Omega
	\end{array}
	\right.
\end{equation}
and 
\begin{equation}\label{pop}
	\left\{
	\begin{array}{rllll}
		\dis -p_{t}-\Delta p+\int_\Omega K(t,x,\theta)p(t,\theta)\ d\theta &=&(z-z_d)\chi_{\O_d}& \mbox{in}& Q,\\
		\dis p&=&0& \mbox{on}& \Sigma, \\
		\dis  p(T,\cdot)&=&0 &\mbox{in}&\Omega.
	\end{array}
	\right.
\end{equation}
\end{proposition}

 \begin{remark}
 		
 	 Using Proposition $2.2$ \cite{djomegnebackward}, we can prove that there exists a positive constant $C=C(\|K\|_{\infty}, T,\mu)$ such that the following estimate holds
 		\begin{equation}\label{v}
 		\|\hat{v}\|_{L^2((0,T)\times \O)}\leq 
 		C\left(
 		\|f\|_{L^2((0,T)\times\omega)}+\|z^0\|_{L^2(\Omega)}\right).
 		\end{equation} 	
 \end{remark}

\subsection{Carleman inequalities}\label{carleman}
In this section, we establish a suitable observability inequality useful for solving the null controllability of the cascade  linear system \eqref{yop}-\eqref{pop}. It is well known by now that, solve the null controllability of the system \eqref{yop}-\eqref{pop} can be reduced to find an observability inequality for the solutions of the following adjoint systems:
\begin{equation}\label{rhoadj}
	\left\{
	\begin{array}{rllll}
		\dis -\rho_t-\Delta \rho+\int_\Omega K(t,x,\theta)\rho(t,\theta)\ d\theta &=&\dis\psi\chi_{\O_d}& \mbox{in}& Q,\\
		\dis \rho&=&0& \mbox{on}& \Sigma,\\
		\dis \rho(T,\cdot)&=&\rho^T& \mbox{in}& \Omega
	\end{array}
	\right.
\end{equation}
and
\begin{equation}\label{psiadj}
	\left\{
	\begin{array}{rllll}
		\dis \psi_t-\Delta \psi+\int_\Omega K(t,x,\theta)\psi(t,\theta)\ d\theta &=&-\dis\frac{1}{\mu}\rho\chi_{\O}& \mbox{in}& Q,\\
		\dis \psi&=&0& \mbox{on}& \Sigma, \\
		\dis \psi(0,\cdot)&=&0&\mbox{in}&\Omega,
	\end{array}
	\right.
\end{equation}
for $\rho^T\in L^2(\Omega)$.

Now, we introduce some notations and show some preliminary results that will be useful in the sequel.

 Since $\O_d\cap\omega\neq \emptyset$, then there exists a nonempty set $\omega^\prime$ such that $\omega^\prime\subset\subset \O_d\cap\omega$. Let $\eta^0\in \mathcal{C}^2(\overline{\Omega})$ be a function satisfying
\begin{equation}\label{c3}
\eta^0(x)>0	\ \mbox{in}\ \Omega,\ \eta^0=0	\ \mbox{on}\ \partial\Omega,\  |\nabla \eta^0|>0\ \mbox{in}\ \overline{\Omega}\setminus\omega^\prime.
\end{equation} 
The existence of such a function is proved in \cite[Lemma 1.1]{FursikovImanuvilov1996}. Let $l\in \mathcal{C}^\infty([0,T])$ be a function satisfying
\begin{equation}\label{l}
	l(t)>0,\ \forall t\in (0,T),\ \ l(t)=t,\ \forall t\in [0,T/4],\ \ l(t)=T-t,\ \forall t\in [3T/4,T].
\end{equation}
For any $\lambda\geq 1$ and for any $(t,x)\in Q$, we define the following (positive) weight functions
\begin{equation}\label{eta}
	\eta(t,x)=\frac{\sigma(x)}{l^4(t)},\ \ 	\varphi(t,x)=\frac{e^{\lambda(2\|\eta^0\|_{L^\infty(\Omega)}+\eta^0(x))}}{l^4(t)},
\end{equation}
where
\begin{equation}
\sigma(x)=e^{4\lambda\|\eta^0\|_{L^\infty(\Omega)}}-e^{\lambda(2\|\eta^0\|_{L^\infty(\Omega)}+\eta^0(x))}.
\end{equation}
In the sequel, we use the following notation defined as in \cite{biccari2019null}:
\begin{subequations}
	\begin{alignat}{11}
		\sigma^+&:=&\dis  \max_{x\in \overline{\Omega}}\sigma(x)=e^{4\lambda\|\eta^0\|_{L^\infty(\Omega)}}-e^{2\lambda\|\eta^0\|_{L^\infty(\Omega)}}\label{sigmaplus},\\
			\sigma^-&:=&\dis  \min_{x\in \overline{\Omega}}\sigma(x)=e^{4\lambda\|\eta^0\|_{L^\infty(\Omega)}}-e^{3\lambda\|\eta^0\|_{L^\infty(\Omega)}}\label{sigmamoins}.
	\end{alignat}
\end{subequations}
Using the definition of $\varphi$, $\eta$ and with a simple computation, we show the following estimates
\begin{equation}\label{gradien}
\nabla\eta=-\nabla\varphi\leq C\lambda\varphi,\ \ |\eta_t|=|\varphi_t|\leq C(T)\varphi^2,\ \ \varphi^{-1}\leq C(T). 	
\end{equation}

For $s\geq 1$, we introduce for a suitable function $z$, the following notation to abridge the estimate
\begin{equation}\label{car2}
	\mathcal{M}(z)=s\lambda^2\int_{Q}e^{-2s\eta}\varphi|\nabla z|^2\dT+s^3\lambda^4\int_{Q}e^{-2s\eta}\varphi^3| z|^2\dT.
\end{equation}
For any $F\in L^2(Q)$ and $z^T\in L^2(\Omega)$, we consider the following system:
\begin{equation}\label{c6}
	\left\{
	\begin{array}{rllll}
		\dis -z_t-\Delta z&=&F& \mbox{in}& Q,\\
		\dis z&=&0& \mbox{on}& \Sigma,\\
		\dis z(T,\cdot)&=&z^T& \mbox{in}& \Omega.
	\end{array}
	\right.
\end{equation}

We state the classical Carleman estimate due to \cite{FursikovImanuvilov1996}  for the solutions to the heat equation \eqref{c6}.
\begin{lemma} \label{prop4}
 Let $\varphi$ and $\eta$ be defined by \eqref{eta}. Then, there exist  positive constants $s_1\geq 1$, $\lambda_1\geq 1$ and $C(\Omega,\omega)>0$ such that, for any $\lambda\geq \lambda_1$, $s\geq s_1$, and for any $z$ solution of \eqref{c6}, we have
	\begin{equation}\label{car}
		\begin{array}{llll}
			\dis \mathcal{M}(z)\leq 
			\dis C(\Omega,\omega)
			\left(\int_{Q}e^{-2s\eta}|F|^2\dT+s^3\lambda^4\int_{0}^{T}\int_{\omega^\prime}e^{-2s\eta}\varphi^3|z|^2\dT
			\right).	
		\end{array}
	\end{equation}	
\end{lemma}

\begin{remark}\label{rem3}
	If we make a change of variable $t$ for $T-t$ in \eqref{c6},
	we have
	\begin{equation}\label{eqctilde}
		\left\{
		\begin{array}{rllll}
			\dis \frac{\partial{\tilde{z}}}{\partial{t}}-\Delta \tilde{z} &=&\widetilde{F}& \mbox{in}& Q,\\
			\dis \tilde{z}&=&0& \mbox{on}& \Sigma,\\
			\dis \tilde{z}(0,\cdot)&=&z^T& \mbox{in}& \Omega,
		\end{array}
		\right.
	\end{equation}
	where $\tilde{z}(t,x)=z(T-t,x)$ and  $\widetilde{F}(t,x)=F(T-t,x)$. Then, the global Carleman inequality \eqref{car} is also valid for any $\tilde{z}$ solution of \eqref{eqctilde}.
\end{remark}
Now, we state the following result having a similar proof to that of the Proposition $2.3$ \cite{biccari2019null}.
\begin{proposition}\label{exp}
	For any fixed $\lambda> 0$ and $s>1$, the following inequality holds
	\begin{equation}
	\exp\left(-\frac{(1+s)\sigma^-}{l^4(t)}\right)<\exp\left(-\frac{s\sigma^+}{l^4(t)}\right).
	\end{equation}
\end{proposition}
\begin{remark}
	In \cite{biccari2019null, FursikovImanuvilov1996}, the authors used the function $l(t)=t(T-t)$ to prove Lemma \ref{prop4} and Proposition \ref{exp}, respectively. Since in this paper we propose a positive function $l(t)$ defined by \eqref{l} with the property of going to $0$ as $t$ tends to $0^+$ and $T^-$, the proof of Lemma \ref{prop4} and Proposition \ref{exp} doesn't change if we use the weight functions defined by \eqref{eta} and the function $l(t)$ defined by \eqref{l}.
\end{remark}

Now, using the estimate \eqref{car}, we prove a new Carleman estimate for the solutions to system \eqref{rhoadj}-\eqref{psiadj}.

\begin{proposition} \label{propcarl} 
	 Assume  that $\omega\cap\O_d\neq \emptyset$ and $\mu$ is large enough. Then, there exists a  positive constant $C=C(\Omega, \omega,T)>0$ such that the solution $(\rho,\psi)$ to \eqref{rhoadj}-\eqref{psiadj} satisfies
	\begin{equation}\label{roba}
		\begin{array}{lll}
			\dis \mathcal{M}(\rho)+\mathcal{M}(\psi)\leq
			\dis Cs^7\lambda^8\int_0^T\int_{\omega}e^{-2s\eta}\varphi^7|\rho|^2\dT,
		\end{array}
	\end{equation}		
	for every $s>0$ large enough.
\end{proposition}

\begin{proof}
	We proceed in two steps. \\
	\textbf{Step 1.} We prove that for $s>0$ sufficiently large, we have
	\begin{equation}\label{rob}
		\begin{array}{lll}
			\dis \mathcal{M}(\rho)+\mathcal{M}(\psi)\leq
			\dis C(\Omega, \omega)\left (s^3\lambda^4\int_{0}^{T}\int_{\omega^\prime}e^{-2s\eta}\varphi^3\left(|\rho|^2
			+|\psi|^2\right)\dT\right).
		\end{array}
	\end{equation}	
	We apply  Carleman inequality \eqref{car} to each equation of \eqref{rhoadj} and \eqref{psiadj} and add them up, we have
	\begin{equation*}
			\begin{array}{lll}
			\mathcal{M}(\rho)+\mathcal{M}(\psi)	&\leq&\dis  C(\Omega, \omega)s^3\lambda^4 \int_0^T\int_{\omega^\prime}e^{-2s\eta}\varphi^3\left(|\rho|^2+|\psi|^2\right)\dT\\
			&+&\dis C(\Omega, \omega)\left( \int_{Q}e^{-2s\eta}|\psi|^2\dT+ \int_{Q}e^{-2s\eta}\left|-\frac{1}{\mu}\rho\chi_{\O}\right|^2\dT\right)\\
			&+&\dis C(\Omega, \omega) \int_{Q}e^{-2s\eta}\left|\int_\Omega K(t,x,\theta)\rho(t,\theta)\ d\theta\right|^2\dT\\
			&+&\dis C(\Omega, \omega)\int_{Q}e^{-2s\eta}\left|\int_\Omega K(t,x,\theta)\psi(t,\theta)\ d\theta\right|^2\dT.
		\end{array}
	\end{equation*}
	Taking the parameter $s$ large enough, we can absorb the second expression in the right-hand side of above expression and we obtain
	\begin{equation}\label{term}
		\begin{array}{lll}
			\mathcal{M}(\rho)+\mathcal{M}(\psi)	&\leq&\dis  C(\Omega, \omega)s^3\lambda^4 \int_0^T\int_{\omega^\prime}e^{-2s\eta}\varphi^3\left(|\rho|^2+|\psi|^2\right)\dT\\
			&+&\dis C(\Omega, \omega) \int_{Q}e^{-2s\eta}\left|\int_\Omega K(t,x,\theta)\rho(t,\theta)\ d\theta\right|^2\dT\\
			&+&\dis C(\Omega, \omega)\int_{Q}e^{-2s\eta}\left|\int_\Omega K(t,x,\theta)\psi(t,\theta)\ d\theta\right|^2\dT.
		\end{array}
	\end{equation}
Now, arguing as in \cite{biccari2019null}, we estimate the two last terms in \eqref{term}. We fix the parameter $\lambda=\lambda_1$	sufficiently large. We have
\begin{equation*}\label{}
	\begin{array}{lll}
	\dis \left|\int_\Omega K(t,x,\theta)\rho(t,\theta)\ d\theta\right|&=&\dis \left|\int_\Omega e^{\frac{\sigma^-}{l^4(t)}}K(t,x,\theta)e^{-\frac{\sigma^-}{l^4(t)}}\rho(t,\theta)\ d\theta\right|\\
	&\leq&\dis \left[\left(\int_\Omega e^{\frac{2\sigma^-}{l^4(t)}}|K(t,x,\theta)|^2\ d\theta\right)\left(\int_\Omega e^{-\frac{2\sigma^-}{l^4(t)}}|\rho(t,\theta)|^2\ d\theta\right)\right]^{1/2}.
	\end{array}
\end{equation*}
Using the above expression, one get
\begin{equation}\label{est}
	\int_{Q}e^{-2s\eta}\left|\int_\Omega K(t,x,\theta)\rho(t,\theta)\ d\theta\right|^2\dT\leq \mathcal{K}^2  \int_{Q}e^{-2s\eta}\left(\int_\Omega e^{-\frac{2\sigma^-}{l^4(t)}}|\rho(t,\theta)|^2\ d\theta\right)\ dxdt.
\end{equation}	
We proceed to estimate the term in the right-hand side of the above inequality. We use Fubini's theorem and we obtain 
\begin{equation*}\label{}
	\begin{array}{lll}
		\dis \int_{Q}e^{-2s\eta}\left(\int_\Omega e^{-\frac{2\sigma^-}{l^4(t)}}|\rho(t,\theta)|^2\ d\theta\right)\ dxdt=\dis \int_{Q} e^{-\frac{2\sigma^-}{l^4(t)}}|\rho(t,\theta)|^2 \left(\int_\Omega e^{-2s\eta}\ dx\right)\ d\theta dt.
	\end{array}
\end{equation*}
Using the definition of $\sigma^-$ given by \eqref{sigmamoins}, we get
\begin{equation}\label{measure}
\int_\Omega e^{-2s\eta}\ dx\leq|\Omega|e^{-\frac{2s\sigma^-}{l^4(t)}},	
\end{equation}
where $|\Omega|$ denotes the measure of $\Omega$. Due to Proposition \ref{exp}, the definition of $\sigma^+$ given by \eqref{sigmaplus} and the estimate \eqref{measure}, we have
\begin{equation*}\label{}
	\begin{array}{lll}
		\dis \int_{Q}e^{-2s\eta}\left(\int_\Omega e^{-\frac{2\sigma^-}{l^4(t)}}|\rho(t,\theta)|^2\ d\theta\right)\ dxdt&\leq&C\dis \int_{Q} e^{-\frac{2(1+s)\sigma^-}{l^4(t)}}|\rho(t,\theta)|^2\ d\theta\ dt\\
		&\leq&\dis C\dis \int_{Q} e^{-\frac{2s\sigma^+}{l^4(t)}}|\rho(t,\theta)|^2\ d\theta\ dt\\
		&\leq&\dis C\dis \int_{Q} e^{-2s\eta}|\rho(t,\theta)|^2\ d\theta\ dt.
	\end{array}
\end{equation*}
Putting it all together, the expression \eqref{est} becomes
\begin{equation}\label{esta}
	\int_{Q}e^{-2s\eta}\left|\int_\Omega K(t,x,\theta)\rho(t,\theta)\ d\theta\right|^2\dT\leq C\mathcal{K}^2  \int_{Q} e^{-2s\eta}|\rho|^2\ dxdt.
\end{equation}	
Proceeding as above, we prove the following estimate
\begin{equation}\label{estb}
	\int_{Q}e^{-2s\eta}\left|\int_\Omega K(t,x,\theta)\psi(t,\theta)\ d\theta\right|^2\dT\leq C\mathcal{K}^2  \int_{Q} e^{-2s\eta}|\psi|^2\ dxdt.
\end{equation}	
Combining \eqref{esta} and \eqref{estb} and substituting in \eqref{term}, then taking $s$ sufficiently large, we deduce estimate \eqref{rob}.\\
	\textbf{Step 2. } Now, we want to eliminate the local term corresponding to $\psi$  on the right hand side of the estimate \eqref{rob}. We argue as the same as in \cite{teresa2000insensitizing}. Consider an open set $\omega_0$ such that $\omega^\prime\subset\subset \omega_0\subset\subset\O_d\cap\omega$ and a function $\xi\in C^{\infty}_0(\Omega)$ satisfying
	 \begin{subequations}\label{owogene1}
		\begin{alignat}{11}
			\dis 	0\leq \xi\leq 1,\,\,\xi=1 \hbox{ in } \omega^\prime ,\,\,  \xi=0 \hbox{ in } \Omega\setminus\omega_0,\label{owo21}\\
			\dis 	\frac{\Delta\xi}{\xi^{1/2}}\in L^{\infty}(\omega_0),\,\, \frac{\nabla\xi}{\xi^{1/2}}\in [L^{\infty}(\omega_0)]^N.\label{owo2}
		\end{alignat}
	\end{subequations}
	We define $u=s^3\lambda^4\varphi^3e^{-2s\eta}$. Using the definition of $\varphi$ and $\eta$, we get $u(T)=u(0)=0$. Due to estimates \eqref{gradien}, the following estimations hold:
	\begin{equation}\label{con}
		\begin{array}{rll}
			\dis |u\xi|\leq s^3\lambda^4\varphi^3e^{2s\eta}\xi,\ \ \ \ \ \
			\dis \left|(u\xi)_t\right|\leq C(T)s^4\lambda^4\varphi^5e^{2s\eta}\xi,\\
			\\ 
			\dis |\nabla(u\xi)|\leq Cs^4\lambda^5\varphi^4e^{2s\eta}\xi,\ \ \ \ \ \
			\dis |\Delta(u\xi)|\leq Cs^5\lambda^6\varphi^5e^{2s\eta}\xi,
		\end{array}
	\end{equation}
	where $C$ is a positive constant.\\ 
	Multiplying the first equation of \eqref{rhoadj}  by $u\xi\psi$, integrating by parts over $Q$ and using Fubini's theorem, we obtain
		\begin{equation}\label{beau}
		J_1+J_2+J_3+J_4=\int_{Q}u\xi|\psi|^2\chi_{\O_d}\ \dT,
	\end{equation}
	where
	$$
	\begin{array}{lllll}
		J_1=\dis-\frac{1}{\mu}\int_{Q}u\xi|\rho|^2\chi_{\O} \ \dT \ \dT,\,
		J_2=\int_{Q}\rho\psi(u\xi)_t\ \dT,\\
		J_3= \dis -2\int_{Q}\rho\nabla\psi.\nabla(u\xi)\ \dT,\ 
		\dis J_4=\dis-\int_Q\rho\psi\Delta(u\xi)\ \dT.
	\end{array}
	$$
	Let us estimate each $J_i,\ i=1,\cdots,4$. From the H\"{o}lder's and Young's inequalities and estimates \eqref{con}, we have
	\begin{equation*}
		J_1
		\leq\dis \frac{1}{\mu}\int_{Q}s^3\lambda^4\varphi^3e^{-2s\eta}|\rho|^2 \ \dT.
	\end{equation*}
	\begin{equation*}
		J_2
		\leq\frac{\delta_1}{2}\int_{Q}  s^3\lambda^4\varphi^3e^{-2s\eta}|\psi|^2\dT+\frac{C(T)}{\delta_1}\int_{0}^{T}\int_{\omega_0} s^5\lambda^4\varphi^7e^{-2s\eta}|\rho|^2\dT,
	\end{equation*}
	for any $\delta_1>0$.
	\begin{equation*}
		J_3
		\leq \frac{\delta_2}{2}\int_{Q} s\lambda^2\varphi e^{-2s\eta}|\nabla \psi|^2\dT+\frac{C}{\delta_2}\int_{0}^{T}\int_{\omega_0} s^7\lambda^8\varphi^7e^{-2s\eta}|\rho|^2\dT,	
	\end{equation*}
for any $\delta_2>0$.
	\begin{equation*}
		J_4\leq \frac{\delta_3}{2}\int_{Q}s^3\lambda^4\varphi^3 e^{-2s\eta}|\psi|^2\dT+\frac{C}{\delta_3}
		\int_{0}^{T}\int_{\omega_0} s^7\lambda^8\varphi^7e^{-2s\eta}|\rho|^2\dT,
	\end{equation*}
	for any $\delta_3>0$.
	
	Finally, choosing the constants $\delta_i$ such that $\dis \delta_2=\frac{1}{C(\Omega,\omega)}$ and $\dis \delta_1=\delta_3=\frac{1}{2C(\Omega,\omega)}$, where $C(\Omega,\omega)$ is the constant obtained in \eqref{rob}, it follows from \eqref{beau}, the previous inequalities and \eqref{con} that
	\begin{equation}\label{owo4}
		\begin{array}{rll}
		\dis \int_{0}^{T}\int_{\omega^\prime} s^3\lambda^4\varphi^3e^{-2s\eta}|\psi|^2\dT\dis &\leq& \dis \frac{1}{2}\mathcal{M}(\psi)+\frac{1}{\mu}\int_{Q}s^3\lambda^4\varphi^3e^{-2s\eta}|\rho|^2 \ \dT\\
		&+&\dis C(\Omega,\omega,T)\int_0^T\int_{\omega_0}s^7\lambda^8\varphi^7e^{-2s\eta}|\rho|^2 \ \dT.
		\end{array}
	\end{equation}
	Combining \eqref{rob} with \eqref{owo4}, we deduce that 
	$$
	\begin{array}{lll}
		\dis 	\mathcal{M}(\rho)+\mathcal{M}(\psi)\leq\dis C(\Omega,\omega,T)\int_0^T\int_{\omega_0}s^7\lambda^8\varphi^7e^{-2s\eta}|\rho|^2 \ \dT+ \frac{1}{\mu}\int_{Q}s^3\lambda^4\varphi^3e^{-2s\eta}|\rho|^2 \ \dT.
	\end{array}	
	$$
	Since $\mu$ is sufficiently large, we can absorb the last term of the above inequality in the left hand side. Using the fact that $\omega_0\subset \omega$, we deduce \eqref{roba}.
\end{proof}
	
\subsection{Observability and null controllability}\label{obsernull}

To prove the needed observability inequality, we are going to improve the Carleman inequality \eqref{roba}. To this end, we introduce some new
weight functions, similar to the ones in \eqref{eta}, that do not vanish at $t =0$ 

Let us consider firstly the function
\begin{equation}
	\bar{l}(t)=\left\{
	\begin{array}{lllll}
		\dis\|l\|_{L^\infty((0,T))}\ \ \mbox{for}\ \ 0\leq t\leq T/2,\\
		\\
		\dis l(t)\ \ \mbox{for}\ \ T/2\leq t\leq T,
	\end{array}
	\right.
\end{equation}
and the functions
\begin{equation}\label{alpha}
	\begin{array}{rll}
		&&\dis	\alpha(t,x)=\frac{e^{4\lambda\|\eta^0\|_{L^\infty(\Omega)}}-e^{\lambda(2\|\eta^0\|_{L^\infty(\Omega)}+\eta^0(x))}}{\bar{l}^4(t)},\ \ 	\zeta(t,x)=\frac{e^{\lambda(2\|\eta^0\|_{L^\infty(\Omega)}+\eta^0(x))}}{\bar{l}^4(t)},\\
		\\
	&&\dis \alpha^*(t)=\dis  \max_{x\in \overline{\Omega}}\alpha(t,x),\ \ \ \zeta^*(t)=\dis  \min_{x\in \overline{\Omega}}\zeta(t,x).	
	\end{array}
\end{equation}
With these definitions, we have the following result.

\begin{proposition} \label{pro}
	  Under the assumptions of Proposition \ref{propcarl}, there exist a positive constant $C=C(\Omega,\omega, \|K\|_{\infty}, T)$ and a positive weight function $\varpi_1$ such that 
	\begin{equation}\label{obser3}
	\|\rho(0,\cdot)\|^2_{L^2(\Omega)}+\int_{Q}\frac{1}{\varpi_1^2}|\psi|^2\,\dT
	\leq C\int_{0}^{T}\int_{\omega}|\rho|^2\,\dq,
	\end{equation}
where $(\rho, \psi)$ is the solution of \eqref{rhoadj}-\eqref{psiadj}.
\end{proposition}

\begin{proof} 
We fix the parameter $s$ large enough and we proceed in two steps. \\
	\textbf{Step 1.} By construction, $\eta=\alpha$ and $\varphi=\zeta$ in $[T/2,T]\times \Omega$, hence, we obtain
	\begin{equation}\label{pou5}
		\begin{array}{llll}
			\dis \int_{T/2}^{T}\int_{\Omega}e^{-2s\alpha}\zeta^3|\rho|^2\ dxdt+\int_{T/2}^{T}\int_{\Omega}e^{-2s\alpha}\zeta^3|\psi|^2\ dxdt\\
			=\dis \int_{T/2}^{T}\int_{\Omega}e^{-2s\eta}\varphi^3|\rho|^2\ dxdt+\int_{T/2}^{T}\int_{\Omega}e^{-2s\eta}\varphi^3|\psi|^2\ dxdt\\
			\leq\dis C(\Omega, \omega,T)\int_{0}^{T}\int_{\omega}e^{-2s\eta}\varphi^7|\rho|^2\,\dq,
		\end{array}
	\end{equation}
where we have used the Carleman estimate \eqref{roba}.\\
\textbf{Step 2.}
On the other hand, we use the classical energy estimate	 for the system \eqref{rhoadj} on the domain $(0,T/2)\times \Omega$. Let us introduce a function $\beta\in \mathcal{C}^1 ([0,T])$  such that
	\begin{equation}\label{hypobeta}
	0\leq \beta\leq 1,\ \beta(t)=1 \hbox{ for } t\in [0,T/2],\  \beta(t)=0 \hbox{ for } t\in [3T/4,T],\  |\beta^\prime(t)|\leq C/T.
	\end{equation}
	For any $(t,x)\in Q$, we set
	$$
	\begin{array}{lll}
	z(t,x)=\beta(t)e^{-r(T-t)}\rho(t,x),
	\end{array}
	$$
	where $r>0$. Then in view of \eqref{rhoadj}, the function $z$ is solution of
	\begin{equation}\label{eq26z}
		\left\{
		\begin{array}{rllll}
			\dis -z_{t}-\Delta z+\int_\Omega K(t,x,\theta)z(t,\theta)\ d\theta +rz  &=& \dis \beta e^{-r(T-t)}\psi\chi_{\O_d}-\beta^\prime e^{-r(T-t)} \rho & \mbox{in}& Q,\\
			\dis z&=&0& \mbox{on}& \Sigma, \\
			\dis z(T,\cdot)&=&0 &\mbox{in}&\Omega.
		\end{array}
		\right.
	\end{equation}
	From the classical energy estimate for the solution $z$ to system \eqref{eq26z} and using the definition of $\beta$ and $z$, we obtain
	$$
	\begin{array}{llll}
	\dis \|\rho(0,\cdot)\|^2_{L^2(\Omega)}+\int_{0}^{T/2}\int_{\Omega}|\rho|^2\ dxdt\dis \leq C(\|K\|_{\infty},T)
	\left(\int_0^{3T/4}\int_\Omega |\psi|^2\ \dq
	+\int_{T/2}^{3T/4}\int_\Omega|\rho|^2\ \dq\right).
	\end{array}
	$$
	Adding the term $\dis \int_{0}^{T/2}\int_\Omega|\psi|^2\ \dq$ on both sides of above estimate, we have
	\begin{equation}\label{pou1}
	\begin{array}{llll}
		\dis \|\rho(0,\cdot)\|^2_{L^2(\Omega)}+\int_{0}^{T/2}\int_{\Omega}|\rho|^2\ dxdt+\int_{0}^{T/2}\int_\Omega|\psi|^2\ \dq\\
		\dis \leq C(\|K\|_{\infty},T)
		\left(\int_{T/2}^{3T/4}\int_\Omega |\psi|^2\ \dq
		+\int_{T/2}^{3T/4}\int_\Omega|\rho|^2\ \dq+\int_{0}^{T/2}\int_\Omega|\psi|^2\ \dq\right).
	\end{array}	
	\end{equation}
	In order to eliminate the last term in the right hand side of \eqref{pou1}, we use the classical energy estimates for the system \eqref{psiadj} and we obtain:
	\begin{equation}\label{pou}
		\begin{array}{llll}
			\dis \int_{0}^{T/2}\int_{\Omega}|\psi|^2\ dxdt
			\dis \leq \frac{1}{\mu^2}C(\|K\|_{\infty}, T) \int_{0}^{T/2}\int_\Omega |\rho|^2\ dxdt.
		\end{array}
	\end{equation}
	Replacing \eqref{pou} in \eqref{pou1} and taking $\mu$ large enough, we obtain
	\begin{equation*}
		\begin{array}{llll}
			\dis \|\rho(0,\cdot)\|^2_{L^2(\Omega)}+\int_{0}^{T/2}\int_{\Omega}|\rho|^2\ dxdt+\int_{0}^{T/2}\int_\Omega|\psi|^2\ \dq\\
			\dis \leq C(\|K\|_{\infty},T)
			\left(\int_{T/2}^{3T/4}\int_\Omega |\psi|^2\ \dq
			+\int_{T/2}^{3T/4}\int_\Omega|\rho|^2\ \dq\right).
		\end{array}	
	\end{equation*}
Using the definition \eqref{alpha}, the weight functions $\alpha$ and $\zeta$ have lower and upper bounds for $(t,x)\in [0,3T/4]$, then we obtain 
	\begin{equation}\label{}
	\begin{array}{llll}
		\dis \|\rho(0,\cdot)\|^2_{L^2(\Omega)}+\int_{0}^{T/2}\int_{\Omega}e^{-2s\alpha}\zeta^3|\rho|^2\ dxdt+\int_{0}^{T/2}\int_\Omega e^{-2s\alpha}\zeta^3|\psi|^2\ \dq\\
		\dis \leq C(\|K\|_{\infty},T)
		\left(\int_{T/2}^{3T/4}\int_\Omega e^{-2s\alpha}\zeta^3|\psi|^2\ \dq
		+\int_{T/2}^{3T/4}\int_\Omega e^{-2s\alpha}\zeta^3|\rho|^2\ \dq\right).
	\end{array}	
\end{equation}	
Using estimate \eqref{pou5}, the previous estimate becomes
	\begin{equation}\label{pou3}
		\begin{array}{llll}
			\dis \|\rho(0,\cdot)\|^2_{L^2(\Omega)}+\int_{0}^{T/2}\int_{\Omega}e^{-2s\alpha}\zeta^3|\rho|^2\ dxdt+\int_{0}^{T/2}\int_\Omega e^{-2s\alpha}\zeta^3|\psi|^2\ \dq\\
			\dis \leq 
			C(\Omega, \omega,\|K\|_{\infty},T)\int_{0}^{T}\int_{\omega}e^{-2s\eta}\varphi^7|\rho|^2\,\dq.
		\end{array}	
	\end{equation}	
Adding estimates \eqref{pou5} and \eqref{pou3}, and using the definitions of $\alpha^*$ and $\zeta^*$ given by \eqref{alpha}, we have	
\begin{equation}
	\begin{array}{llll}
		\dis \|\rho(0,\cdot)\|^2_{L^2(\Omega)}+\int_{0}^{T}\int_{\Omega}e^{-2s\alpha^*}(\zeta^*)^3|\rho|^2\ dxdt+\int_{0}^{T}\int_\Omega e^{-2s\alpha^*}(\zeta^*)^3|\psi|^2\ \dq\\
		\dis \leq 
		C(\Omega, \omega,\|K\|_{\infty},T)\int_{0}^{T}\int_{\omega}e^{-2s\eta}\varphi^7|\rho|^2\,\dq.
	\end{array}	
\end{equation}	
We obtain the observability inequality \eqref{obser3} due to the fact that $\dis e^{-2s\eta}\varphi^7\in L^\infty(Q)$ and taking 
\begin{equation}\label{varpi}
	\varpi_1(t)=e^{s\alpha^*}(\zeta^*)^{-3/2}.
\end{equation}
\end{proof}

Now, we end the proof of Theorem \ref{theolinear} by proving the null controllability of system \eqref{yop}-\eqref{pop}. Thanks to the Proposition \ref{pro}, the following result holds.

\begin{proposition}\label{linear}
	Assume that $\O_d\cap\omega\neq \emptyset$. Let $\varpi_1$ be defined as in Proposition \ref{pro}. Let also $z^0\in L^2(\Omega)$ and $z_{d}\in L^2((0,T)\times \O_{d})$ such that \eqref{nou} holds. Then, there exists a leader control $\hat{f}\in L^2((0,T)\times\omega)$ such that the associated solution $(z, p)$ to \eqref{yop}-\eqref{pop} satisfies \eqref{mainobj}. Furthermore; there exists a constant $C=C(\Omega,\omega, \|K\|_{\infty}, T)$ such that
	\begin{equation}\label{mon10theo}
		\begin{array}{ccc}
			\dis \|\hat{f}\|_{L^2((0,T)\times\omega)}\leq C
			\dis \left( \|\varpi_1 z_{d}\|_{L^2((
				0,T)\times\O_{d})}+ \|y^0\|_{L^2(\Omega)}\right).
		\end{array}
	\end{equation}
\end{proposition}

\begin{proof}
Let $\varepsilon >0$ and consider the following cost function
\begin{equation}\label{defJ}
	J_{\varepsilon}(f)=\dis \frac{1}{2\varepsilon}\int_{\Omega}|z(T,x)|^2\ dx+
	\frac{1}{2}\int_0^T\int_{\omega}|f|^2\ \dq,
\end{equation}
where $z$ is the solution to system \eqref{yop}-\eqref{pop}.
Then, we consider the optimal control problem: 
\begin{equation}\label{opt}
	\inf_{\atop f\in L^2((0,T)\times\omega)}J_{\varepsilon}(f).
\end{equation}
We can prove that $J_{\varepsilon}$ is continuous, coercive and strictly convex. Then, the optimization problem \eqref{opt} admits a unique solution $f_\varepsilon$  and arguing as in \cite{djomegne2021}, we prove that 
\begin{equation}\label{hgeps}
	f_{\varepsilon}=\rho_{\varepsilon}\ \ \mbox{in}\ \ (0,T)\times\omega,
\end{equation}	
with $(\rho_\varepsilon,\psi_\varepsilon)$  solution of the following systems
\begin{equation}\label{rhoeps}
	\left\{
	\begin{array}{rllll}
		\dis -\rho_{\varepsilon,t}-\Delta\rho_{\varepsilon}+\int_\Omega K(t,x,\theta)\rho_{\varepsilon}(t,\theta)\ d\theta  &=&\dis \psi_{\varepsilon}\chi_{\O_{d}}& \mbox{in}& Q,\\
		\dis \rho_{\varepsilon}&=&0& \mbox{on}& \Sigma, \\
		\dis  \rho_{\varepsilon}(T,\cdot)&=&-\dis \frac{1}{\varepsilon}z_{\varepsilon} (T,\cdot) &\mbox{in}&\Omega
	\end{array}
	\right.
\end{equation}
and
\begin{equation}\label{psieps}
	\left\{
	\begin{array}{rllll}
		\dis \psi_{\varepsilon,t}-\Delta\psi_{\varepsilon}+\int_\Omega K(t,x,\theta)\psi_{\varepsilon}(t,\theta)\ d\theta  &=&\dis-\frac{1}{\mu}\rho_{\varepsilon}\chi_{\O}& \mbox{in}& Q,\\
		\dis \psi_{\varepsilon}&=&0& \mbox{on}& \Sigma, \\
		\dis  \psi_{\varepsilon}(0,\cdot)&=&0 &\mbox{in}&\Omega,
	\end{array}
	\right.
\end{equation}
where $(z_\varepsilon, p_\varepsilon)$ is the solution of system \eqref{yop}-\eqref{pop} associated to the control $f_{\varepsilon}$.

Multiplying  the first equation of \eqref{rhoeps} and \eqref{psieps}  by $z_{\varepsilon}$ and $p_{\varepsilon}$ respectively and integrating by parts over $Q$, we obtain from \eqref{hgeps} and the Young's inequality
\begin{equation}\label{rosnymon10}
	\begin{array}{llll}
		\dis \|f_\varepsilon\|^2_{L^2((0,T)\times\omega)}+\frac{1}{\varepsilon}\|z_{\varepsilon}(T,\cdot)\|^2_{L^2(\Omega)}\\
		=\dis-\int_{\Omega}y^0(x)\rho_\varepsilon(0,x)\ dx+
		\int_0^T\int_{\O_d}z_d\psi_{\varepsilon}\ \dT\\
		\leq
		\dis  \left(\left\|\varpi_1 z_d\right\|^2_{L^2((0,T)\times\O_d)}+ \|z^0\|^2_{L^2(\Omega)}\right)^{1/2} \times\left( \left\|\varpi_1^{-1}\psi_{\varepsilon}\right\|^2_{L^2(Q)}+
		\|\rho_{\varepsilon}(0,\cdot)\|^2_{L^2(\Omega)}\right)^{1/2}.
	\end{array}
\end{equation}
Using the observability inequality \eqref{obser3}, we deduce from \eqref{rosnymon10} the existence of a constant $C=C(\Omega,\omega, \|K\|_{\infty}, T)$ such that
\begin{equation}\label{mon10}
	\begin{array}{rll}
		\dis \|f_\varepsilon\|_{L^2((0,T)\times\omega)}\leq C
		\left(\left\|\varpi_1 z_d\right\|_{L^2((0,T)\times\O_d)}+ \|z^0\|_{L^2(\Omega)}\right)
	\end{array}
\end{equation}
and
\begin{equation}\label{mon10b}
	\begin{array}{rll}
		\dis \|z_{\varepsilon}(T,\cdot)\|_{L^2(\Omega)}\leq
		\dis C\sqrt{\varepsilon} \left(\left\|\varpi_1 z_d\right\|_{L^2((0,T)\times\O_d)}+ \|z^0\|_{L^2(\Omega)}\right).
	\end{array}
\end{equation}
Using \eqref{mon10}-\eqref{mon10b} and classical energy estimates to systems \eqref{yop}-\eqref{pop} associated to the control $f_{\varepsilon}$, we can extract subsequences still denoted by $f_\varepsilon,\ z_\varepsilon $ and $p_\varepsilon$ such that when $\varepsilon \rightarrow 0$, one has
\begin{subequations}\label{convergence2}
	\begin{alignat}{9}
		f_\varepsilon&\rightharpoonup& \hat{f}&\text{ weakly in }&L^{2}((0,T)\times\omega), \label{18k}\\
		z_{\varepsilon}&\rightharpoonup& y&\text{  weakly in }&L^2((0,T);H^1_0(\Omega)), \label{}\\
		p_{\varepsilon}&\rightharpoonup& p&\text{  weakly in }&L^2((0,T);H^1_0(\Omega)), \label{19}\\
		z_{\varepsilon}(T,\cdot)&\longrightarrow&0&\hbox{ strongly in }&\ L^2(\Omega).\label{all5}
	\end{alignat}
\end{subequations}

Arguing as in \cite{djomegne2021, djomegne2018} while using convergences \eqref{convergence2}, we prove that $(z,p)$ is a solution of \eqref{yop}-\eqref{pop} corresponding to the control $\hat{f}$ and also $z$ satisfies \eqref{mainobj}. Furthermore, using the convergence \eqref{18k}, we have that $\hat{f}$ satisfies \eqref{mon10theo}.
\end{proof}

\subsection{The semilinear case}\label{nonlin}
In this section, we apply the Stackelberg strategy for the following nonlocal semilinear system
\begin{equation}\label{eqsemi}
	\left\{
	\begin{array}{rllll}
		\dis y_{t}-\Delta y+\int_\Omega K(t,x,\theta)y(t,\theta)\ d\theta &=&G(y)+f\chi_{\omega}+v\chi_{\O}& \mbox{in}& Q,\\
		\dis y&=&0& \mbox{on}& \Sigma, \\
		\dis  y(0,\cdot)&=&y^0 &\mbox{in}&\Omega,
	\end{array}
	\right.
\end{equation}
where $G$ is a globally Lipschitz continuous function. 

In the semilinear case, the convexity of the functional $J$ defined by \eqref{all16} is not guaranteed and we loose the uniqueness of the optimal control $\hat{v}$. In the following proposition, we prove that if $\mu$ is sufficiently large, then the functional \eqref{all16} is indeed strictly convex for small data.
\begin{proposition}\label{convexity}
	
	Let us assume that $y_{d}\in L^\infty((0,T); \O_d)$, $f\in L^2((0,T)\times\omega)$ are fixed and $\mu$ is sufficiently large. Suppose that  $y^0\in L^2(\Omega)$, $N\leq 12$, $G\in\mathcal{C}^2(\R)$ and there exists a constant $L>0$ such that $\dis |G^\prime(s)|+|G^{\prime\prime}(s)|\leq L,\ \forall s\in \R$.
	Then, if $\hat{v}$ satisfies \eqref{Nashprime}, there exists a constant $C>0$ independent of $\mu$ such that 
	\begin{equation}
		\dis 	D^2J(f;\hat{v})\cdot(w^1,w^1)\geq C \|w^1\|^2_{L^2((0,T)\times\O)},\ \forall w^1\in L^2((0,T)\times \O),\ w^1\neq 0.
	\end{equation}
	In particular, the functional $J$ is strictly convex in $ \hat{v}$.	
\end{proposition}

\begin{proof}
The technique of the proof is inspired by \cite[Proposition 1.4]{araruna2015anash}.	
	
Let $f\in L^2((0,T)\times \omega)$ be given and let $\hat{v}$ be the associated optimal control solution of \eqref{Nash1}. For any $w^1,w^2\in L^2((0,T)\times \O)$ and $s\in \R$, we denote by $y^s$ the solution of the following system
\begin{equation}\label{ys}
	\left\{
	\begin{array}{rllll}
		\dis y^s_{t}-\Delta y^s+\int_\Omega K(t,x,\theta)y^s(t,\theta)\ d\theta &=&G(y^s)+f\chi_{\omega}+(\hat{v}+sw^1)\chi_{\O}& \mbox{in}& Q,\\
		\dis y^s&=&0& \mbox{on}& \Sigma, \\
		\dis  y^s(0,\cdot)&=&y^0 &\mbox{in}&\Omega,
	\end{array}
	\right.
\end{equation}
and we set $y:=y^s|_{s=0}$.

Then, we have
\begin{equation}\label{dj}
	\begin{array}{llll}
		&&\dis DJ(f;\hat{v}+sw^1)\cdot w^2-	DJ(f;\hat{v})\cdot w^2=s\mu\int_0^T\int_{\O}w^1w^2\ dxdt\\
		&&\dis +\int_0^T\int_{\O_{d}}(y^s-y_{d})z^{s}\ dxdt-\int_0^T\int_{\O_{d}}(y-y_{d})z\ dxdt,\\
	\end{array}
\end{equation}
where $z^{s}$ is the derivative of $y^s$ with respect to $\hat{v}$ in the direction $w^2$, that is $z^{s}$ is the solution to
\begin{equation}\label{zs}
	\left\{
	\begin{array}{rllll}
		\dis z_{t}^{s}-\Delta z^s+\int_\Omega K(t,x,\theta)z^s(t,\theta)\ d\theta   &=&G^\prime(y^s)+w^2\chi_{\O}& \mbox{in}& Q,\\
		\dis z^{s}&=&0& \mbox{on}& \Sigma, \\
		\dis z^{s}(0,\cdot)&=&0 &\mbox{in}&\Omega.
	\end{array}
	\right.
\end{equation}
We will also use the notation $z:=z^{s}|_{s=0}$.

Let us introduce the adjoint system of \eqref{zs}
\begin{equation}\label{ps}
	\left\{
	\begin{array}{rllll}
		\dis -p_{t}^{s}-\Delta p^s+\int_\Omega K(t,x,\theta)p^s(t,\theta)\ d\theta &=&G^\prime(y^s)p^s+\left(y^s-y_{d}\right)\chi_{\O_{d}}& \mbox{in}& Q,\\
		\dis p^{s}&=&0& \mbox{on}& \Sigma, \\
		\dis p^{s}(T,\cdot)&=&0 &\mbox{in}&\Omega
	\end{array}
	\right.
\end{equation}
and let us use the notation $p:=p^{s}|_{s=0}$.\\
Multiplying the first equation of \eqref{zs} by $p^{s}$ solution of \eqref{ps} and integrating by parts over $Q$, we obtain
\begin{equation}\label{zp1}
	\int_0^T\int_{\O_{d}}(y^s-y_{d})z^{s}\ dxdt=\int_{Q}w^2p^{s}\chi_{\O}\ dxdt.	
\end{equation}
From \eqref{dj} and \eqref{zp1}, we have
\begin{equation}\label{dj1}
	\begin{array}{llll}
		\dis DJ(f;\hat{v}+sw^1)\cdot w^2-	DJ(f;\hat{v})\cdot w^2&=&\dis s\mu\int_0^T\int_{\O}w^1w^2\ dxdt\\
		\dis &+&\dis \int_0^T\int_{\O}(p^{s}-p)w^2\ dxdt.
	\end{array}
\end{equation}Notice that
\begin{eqnarray*}
&&\dis -\frac{\partial}{\partial t}(p^s-p)-\Delta (p^s-p)+\int_\Omega K(t,x,\theta)(p^s(t,\theta)-p^s(t,\theta))\ d\theta\\
\\
&&\dis=\left[G^\prime(y^s)-G^\prime(y)\right]p^s+G^\prime(y)(p^s-p)+(y^s-y)\chi_{\O_{d}}	
\end{eqnarray*}
and
\begin{eqnarray*}
\dis \frac{\partial}{\partial t}(y^s-y)-\Delta (y^s-y)+\int_\Omega K(t,x,\theta)(y^s(t,\theta)-y^s(t,\theta))\ d\theta =\left[G(y^s)-G(y)\right]+sw^1\chi_{\O}.	
\end{eqnarray*}

Consequently, the following limits exist
$$
\eta=\lim\limits_{s\to 0}\frac{1}{s}(p^{s}-p)\ \ \mbox{and}\ \ \phi=\lim\limits_{s\to 0}\frac{1}{s}(y^s-y)
$$
and verify the following systems:
\begin{equation}\label{ha}
	\left\{
	\begin{array}{rllll}
		\dis \phi_{t}-\Delta \phi+\int_\Omega K(t,x,\theta)\phi(t,\theta)\ d\theta &=&G^\prime(y)\phi+w^1\chi_{\O}& \mbox{in}& Q,\\
		\dis \phi&=&0& \mbox{on}& \Sigma, \\
		\dis \phi(0,\cdot)&=&0 &\mbox{in}&\Omega
	\end{array}
	\right.
\end{equation}
and
\begin{equation}\label{eta1}
	\left\{
	\begin{array}{rllll}
		\dis -\eta_{t}-\Delta \eta+\int_\Omega K(t,x,\theta)\eta(t,\theta)\ d\theta   &=&G^\prime(y)\eta+G^{\prime\prime}(y)\phi p+\phi\chi_{\O_{d}}& \mbox{in}& Q,\\
		\dis \eta&=&0& \mbox{on}& \Sigma, \\
		\dis \eta(T,\cdot)&=&0 &\mbox{in}&\Omega.
	\end{array}
	\right.
\end{equation}

Thus, from \eqref{dj1}-\eqref{eta1} with $w^2=w^1$, we have 
\begin{equation}\label{dj3}
	\dis D^2J(f;\hat{v})\cdot(w^1,w^1)= \int_0^T\int_{\O}\eta w^1\ dxdt+\mu\int_0^T\int_{\O}|w^1|^2\ dxdt.
\end{equation}
Now, we want to show that, for some constant $C>0$ independent of $f,\ \eta,\ \phi,\ w^1$ and $\mu$, one has
\begin{equation}\label{a}
	\left|\int_0^T\int_{\O}\eta w^1\ dxdt\right|\leq C\|w^1\|_{L^2((0,T)\times \O)}.
\end{equation}
Indeed, given $w_1\in L^2((0,T)\times\O)$ and since $G^\prime\in L^\infty(Q)$, then, we have the following energy estimate associated to \eqref{ha}
\begin{equation}\label{est1}
	\| \phi\|^2_{L^2(Q)}+\| \nabla\phi\|^2_{L^2(Q)} \leq C\|w^1\|_{L^2((0,T)\times \O)}.
\end{equation}

Using systems \eqref{ha}-\eqref{eta1}, we have 
\begin{equation}\label{fa}
	\begin{array}{llll}
		\dis \int_0^T\int_{\O}\eta w^1\ dxdt&=&\dis \int_Q\eta\left(\phi_t-\Delta \phi+\int_\Omega K(t,x,\theta)\phi(t,\theta)\ d\theta -G^\prime(y)\phi\right)\ dxdt\\
		&=&\dis \int_Q\phi\left(-\eta_t-\Delta \eta+\int_\Omega K(t,x,\theta)\eta(t,\theta)\ d\theta-G^\prime(y)\eta\right)\ dxdt\\
		&=&\dis \int_Q\phi\left(G^{\prime\prime}(y)\phi p+\phi\chi_{\O_d}\right)\ dxdt\\
		&=&\dis \int_Q G^{\prime\prime}(y)|\phi|^2 p\ dxdt+\int_Q|\phi|^2\chi_{\O_d}\ dxdt.
	\end{array}
\end{equation}

Using H\"{o}lder inequality in the last expression of \eqref{fa} and the fact that  $G^{\prime\prime}\in L^\infty(Q)$, we get
\begin{equation}\label{fa1}
	\begin{array}{llll}
		\dis \left|\int_0^T\int_{\O}\eta w^1\ dxdt\right|
		\leq \dis \|G^{\prime\prime}\|_{L^\infty(Q)}\|\phi\|^2_{L^{2r^\prime}((0,T);L^{2s^\prime}(\Omega))} \|p\|_{L^{r}((0,T);L^{s}(\Omega))}
		\dis +\|\phi\|^2_{L^2((0,T)\times\O_d)},
	\end{array}
\end{equation}
where $r^\prime$ and $s^\prime$ are the conjugate of $r$ and $s$, respectively. To bound the right-hand side of this latter inequality, the idea is to find $r$ and $s$ such that 
$$
p\in L^{r}((0,T);L^{s}(\Omega)),\ \ \phi\in L^{2r^\prime}((0,T);L^{2s^\prime}(\Omega)). 
$$
First, we have that $\phi\in L^2((0,T);H^2(\Omega))\cap L^\infty((0,T);H^1(\Omega))$. It is reasonable to look for which values of $d$ and $b$ the following embedding holds:
\begin{equation}
	L^2((0,T);H^2(\Omega))\cap L^\infty((0,T);H^1(\Omega))\hookrightarrow L^{d}((0,T);L^{b}(\Omega)).	
\end{equation}
Using interpolation results, we deduce that
\begin{equation}
	\frac{1}{d}=\frac{\theta}{2},\ \ 0<\theta<1.
\end{equation}
From Sobolev embedding results, we have
\begin{subequations}
	\begin{alignat}{11}
		\dis 	H^2(\Omega)&\hookrightarrow& \dis L^{\frac{2N}{N-4}}(\Omega),\label{h2}\\
		\dis 	H^1(\Omega)&\hookrightarrow&\dis  L^{\frac{2N}{N-2}}(\Omega)\label{h1}.
	\end{alignat}
\end{subequations}
Then, the space $L^{b}(\Omega)$ is an intermediate space with respect to \eqref{h2} and \eqref{h1} if 
\begin{equation}
	\frac{1}{b}=\frac{(N-4)\theta}{2N}+\frac{(N-2)(1-\theta)}{2N},\ \ 0<\theta<1.
\end{equation}
Taking $d=2r^\prime$ and $b=2s^\prime$, it follows that appropriate values of $r$ and $s$ are
$$
r=\frac{d}{d-2}\ \ \mbox{and}\ \ s=\frac{dN}{2(d+2)}.
$$
On the other hand, since $f\in L^2((0,T)\times\omega)$, $v\in L^2((0,T)\times\O)$ and $y^0\in L^2(\Omega)$, we have that 
\begin{equation}
	y\in L^2((0,T);H^1(\Omega))\cap L^\infty((0,T);L^2(\Omega))\hookrightarrow L^{\bar{d}}((0,T);L^{\bar{b}}(\Omega)).	
\end{equation}
Using the interpolation argument, we obtain that 
$$
\bar{b}=\frac{2\bar{d}N}{\bar{d}N-4}.
$$
From parabolic regularity, we have
\begin{equation}
	\phi\in L^{\bar{d}}((0,T);W^{2,\bar{b}}(\Omega))\hookrightarrow L^{\bar{d}}((0,T);L^{\frac{N\bar{b}}{N-2\bar{b}}}(\Omega))=L^{\bar{d}}((0,T);L^{\frac{2\bar{d}N}{\bar{d}(N-4)-4}}(\Omega)).
\end{equation}
Taking, $\bar{d}=r$, it follows that $\dis \phi\in L^{\bar{d}}((0,T);L^{\frac{2dN}{d(N-8)+8}}(\Omega))$ and, in order to have $\phi\in L^{r}((0,T);L^{s}(\Omega))$, we need
$$
\frac{dN}{2(d+2)}\leq \frac{2dN}{d(N-8)+8},
$$
which is true if and only if $N\leq 12$.

Now, using \eqref{est1} and the standard energy estimate of systems \eqref{ys} for $s=0$ and \eqref{ps} for $s=0$, the term  \eqref{fa1} becomes
\begin{equation}\label{a1}
	\left|\int_0^T\int_{\O}\eta w^1\ dxdt\right|\leq C\|w^1\|_{L^2((0,T)\times \O)},
\end{equation}
where $C$ is a positive constant independent of $\mu$.

Combining \eqref{dj3} together with \eqref{a1}, it follows that

\begin{equation*}
	\dis D^2J(f;\hat{v})\cdot(w^1,w^1)\geq\left(\mu-C\right) \int_0^T\int_{\O}|w^1|^2\ dxdt,\ \forall w^1\in L^2((0,T)\times \O),\ w^1\neq 0.
\end{equation*}
For $\mu$ sufficiently large, the functional $J$ given by \eqref{all16} is strictly convex and thus, we have the uniqueness of the optimal control $\hat{v}$ solution of \eqref{Nash1}.	
\end{proof}

Arguing as in the Proposition \ref{quasi}, the optimal control $\hat{v}$ is given by
\begin{equation}\label{vopsemi}
	\hat{v}=-\frac{1}{\mu}p\ \ \mbox{in}\ \ \ (0,T)\times \O,	
\end{equation}
where  $(y,p)$ is the solution to the following optimality system
\begin{equation}\label{ypsemi}
	\left\{
	\begin{array}{ccl}
		\dis y_{t}-\Delta y+\int_\Omega K(t,x,\theta)y(t,\theta)\ d\theta =\dis G(y)+ f\chi_{\omega}-\frac{1}{\mu}p\chi_{\O}& \mbox{in}& Q,\\
		\dis -p_{t}-\Delta p+\int_\Omega K(t,x,\theta)p(t,\theta)\ d\theta =G^\prime(y)p+(y-y_d)\chi_{\O_d}& \mbox{in}& Q,\\
		\dis y=0,\ \ \ p=0& \mbox{on}& \Sigma, \\
		\dis  y(0,\cdot)=y^0,\ \ \ p(T,\cdot)=0 &\mbox{in}&\Omega.
	\end{array}
	\right.
\end{equation}	

In the following result, using the same ideas as in the linear case, we prove the null controllability of the optimality system \eqref{vopsemi}-\eqref{ypsemi}.
\begin{theorem}\label{theosemi}
Under the assumptions of Proposition \ref{convexity}, there exists a positive  weight function $\varpi_1=\varpi_1(t)$ blowing up at $t=T$ such that for any $y_{d}\in L^2((0,T)\times\O_d)$ satisfying \eqref{nou}, there exist a control $\hat{f}\in L^2((0,T)\times \omega)$ and a unique optimal control $\hat{v}\in L^2((0,T)\times \O)$ such that the corresponding solution to \eqref{ypsemi} satisfies \eqref{mainobj}. 
\end{theorem}

\begin{proof}
	We give a sketch of the proof of Theorem \ref{theosemi} which combines the results obtained in the linear case with a fixed point argument.

We introduce the change of variables that allows us to simplify the controllability of trajectories \eqref{mainobj} to null controllability property. By introducing the change of variable $z=y-\bar{y}$, we can write \eqref{ypsemi} as follows
\begin{equation}\label{zpsemi}
	\left\{
	\begin{array}{ccl}
		\dis z_{t}-\Delta z+\int_\Omega K(t,x,\theta)z(t,\theta)\ d\theta =\dis a(z)z+ f\chi_{\omega}-\frac{1}{\mu}p\chi_{\O}& \mbox{in}& Q,\\
		\dis -p_{t}-\Delta p+\int_\Omega K(t,x,\theta)p(t,\theta)\ d\theta =b(z)p+(z-z_d)\chi_{\O_d}& \mbox{in}& Q,\\
		\dis z=0,\ \ \ p=0& \mbox{on}& \Sigma, \\
		\dis  z(0,\cdot)=z^0,\ \ \ p(T,\cdot)=0 &\mbox{in}&\Omega,
	\end{array}
	\right.
\end{equation}
where $z^0:=y^0-\bar{y}^0$, $z_d:=y_d-\bar{y}$ and 
	$$
	a(z)=\int_0^1G^\prime(\bar{y}+s z)\ ds,\ \ b(z)=G^\prime(z+\bar{y}).
	$$
The condition \eqref{mainobj} is equivalent to a null controllability property for $z$, that is
\begin{equation}\label{znullsemi}
	z(T,x)=0\ \mbox{in}\ \Omega.
\end{equation}	

For each $w\in L^2(Q)$, we introduce the linearized system associated to \eqref{zpsemi} 
\begin{equation}\label{zplinear}
	\left\{
	\begin{array}{ccl}
		\dis z_{t}-\Delta z+\int_\Omega K(t,x,\theta)z(t,\theta)\ d\theta =\dis a(w)z+ f\chi_{\omega}-\frac{1}{\mu}p\chi_{\O}& \mbox{in}& Q,\\
		\dis -p_{t}-\Delta p+\int_\Omega K(t,x,\theta)p(t,\theta)\ d\theta =b(w)p+(z-z_d)\chi_{\O_d}& \mbox{in}& Q,\\
		\dis z=0,\ \ \ p=0& \mbox{on}& \Sigma, \\
		\dis  z(0,\cdot)=z^0,\ \ \ p(T,\cdot)=0 &\mbox{in}&\Omega.
	\end{array}
	\right.
\end{equation}
Thanks to the hypothesis on $G$, there exists $M>0$ such that
$$
\|a(w)\|_{L^\infty(Q)},\ \|b(w)\|_{L^\infty(Q)}\leq M,\ \forall w\in L^2(Q).
$$

Arguing as in the proof of Proposition \ref{pro}, we can obtain the following observability inequality
\begin{equation}\label{obsersemi}
	\|\rho(0,\cdot)\|^2_{L^2(\Omega)}+\int_{Q}\frac{1}{\varpi_1^2 }|\psi|^2\,\dT
	\leq C\int_{0}^{T}\int_{\omega}|\rho|^2\,\dq,
\end{equation}
where $C=C(\Omega,\omega, \|K\|_{\infty},\|a(w)\|_{L^\infty(Q)},\|b(w)\|_{L^\infty(Q)}, T)>0$, the function $\varpi_1$ is given by \eqref{varpi} and $(\rho,\psi)$ is the solution to the following adjoint system associated to \eqref{zplinear}:
\begin{equation}\label{}
	\left\{
	\begin{array}{ccl}
		\dis -\rho_t-\Delta \rho+\int_\Omega K(t,x,\theta)\rho(t,\theta)\ d\theta =\dis a(w)\rho+\psi\chi_{\O_d}& \mbox{in}& Q,\\
		\dis \psi_t-\Delta \psi+\int_\Omega K(t,x,\theta)\psi(t,\theta)\ d\theta =b(w)\psi-\dis\frac{1}{\mu}\rho\chi_{\O}& \mbox{in}& Q,\\
		\dis \rho=0,\ \ \ \psi=0& \mbox{on}& \Sigma, \\
		\dis  \rho(T,\cdot)=\rho^T,\ \ \ \psi(0,\cdot)=0 &\mbox{in}&\Omega.
	\end{array}
	\right.
\end{equation}

Thanks to estimate \eqref{obsersemi} and arguing as in the Proposition \ref{linear}, we prove the existence of a leader control $\hat{f}\in L^2((0,T)\times\omega)$ solution to the null controllability problem \eqref{znullsemi}-\eqref{zplinear}. The leader control verifies the following estimate
\begin{equation}\label{leadersemi}
	\begin{array}{ccc}
		\dis \|\hat{f}\|_{L^2((0,T)\times\omega)}\leq C
		\dis \left( \|\varpi_1 z_{d}\|_{L^2((
			0,T)\times\O_{d})}+ \|y^0\|_{L^2(\Omega)}\right),
	\end{array}
\end{equation}	
where $C=C(\Omega,\omega, \|K\|_{\infty},\|a(w)\|_{L^\infty(Q)},\|b(w)\|_{L^\infty(Q)}, T)>0$.

Now, we consider a nonlinear map
$$S:L^2(Q)\rightarrow L^2(Q)$$
such that, for every $w\in L^2(Q)$, $S(w)=y$ where $(y, p)$ are solutions of \eqref{zplinear}. Proving that $S$ has a fixed point $y\in L^2(Q)$  will allows us to prove that $y$ is solution of \eqref{eqsemi}. To this end, we use  the Schauder fixed-point theorem. Proceeding as in \cite[Proposition 3.2]{djomegne2021}, we can prove the following properties for every $w\in L^2(Q)$: 
	\begin{enumerate}
		\item $S$ is continuous,
		\item $S$ is compact,
		\item The range of $S$ is bounded; i.e.,
		$\exists M>0:\ \|S(w)\|_{L^2(Q)}\leq M,\ \forall w\in L^2(Q).$		
	\end{enumerate}
This end the proof of Theorem \ref{theosemi} and furthermore the proof of Stackelberg control of
system \eqref{eqsemi}.
\end{proof}

 \section{Exact controllability with distributed leader and boundary follower}\label{boundary}
In this section, we establish a controllability result for the system \eqref{eqbis}.
Let $z:=q-\bar{q}$, where $q$ and $\bar{q}$ are solutions to \eqref{eqbis} and \eqref{eqbisbar} respectively.  Then, $z$ is the solution of the following system
\begin{equation}\label{qz}
	\left\{
	\begin{array}{rllll}
		\dis z_{t}-\Delta z+\int_\Omega K(t,x,\theta)z(t,\theta)\ d\theta &=&g\chi_{\omega}& \mbox{in}& Q,\\
		\dis z&=&u\textbf{1}_{\Gamma}& \mbox{on}& \Sigma, \\
		\dis  z(0,\cdot)&=&z^0 &\mbox{in}&\Omega,
	\end{array}
	\right.
\end{equation}
where $z^0:=q^0-\bar{q}^0$. Thus, the property \eqref{obj} is equivalent to 
\begin{equation}\label{qnull}
	z(T,x)=0\ \mbox{in}\ \Omega.
\end{equation}
The cost functional $\widetilde{J}$ defined by \eqref{j} can be rewritten as follows
 \begin{equation}\label{jq}
	\widetilde{J}(g;u)=\frac{1}{2}\int_0^T\int_{\O_d}|z-z_d|^2\ dxdt+\frac{\mu}{2}\int_0^T\int_{\Gamma}|u|^2\ d\sigma dt,
\end{equation}
where $z_d:=q_d-\bar{q}$.
	
Adapting the arguments of Section \ref{low}, we can prove the existence and uniqueness of an optimal control $\hat{u}$ solution of \eqref{op}. Indeed, for any fixed leader $g$ and $z^0\in L^2(\Omega)$, one can prove that	
\begin{equation}\label{uop}
	\hat{u}=\frac{1}{\mu}\frac{\partial p}{\partial\nu}\ \ \mbox{on}\ \ \ (0,T)\times \Gamma,	
\end{equation}
where  $(z,p)$ is the solution of optimality system
\begin{equation}\label{zp}
	\left\{
	\begin{array}{ccl}
		\dis z_{t}-\Delta z+\int_\Omega K(t,x,\theta)z(t,\theta)\ d\theta =\dis g\chi_{\omega}& \mbox{in}& Q,\\
		\dis -p_{t}-\Delta p+\int_\Omega K(t,x,\theta)p(t,\theta)\ d\theta =(z-z_d)\chi_{\O_d}& \mbox{in}& Q,\\
		\dis z=\frac{1}{\mu}\frac{\partial p}{\partial\nu}\textbf{1}_{\Gamma},\ \ \ p=0& \mbox{on}& \Sigma, \\
		\dis  z(0,\cdot)=z^0,\ \ \ p(T,\cdot)=0 &\mbox{in}&\Omega.
	\end{array}
	\right.
\end{equation}	

Now, we state a regularity result for the solutions to \eqref{zp}.
\begin{proposition}\label{well}
	
	Assume that $g\in L^2((0,T)\times\omega)$, $z^0\in L^2(\Omega)$, $z_d\in L^2((0,T)\times \O_d)$ and $\mu$ is large enough. Then, the system \eqref{zp} has a unique solution $(z,p)\in W(Q)\times H^{2,1}(Q)$. Moreover, $\dis \frac{\partial p}{\partial\nu}\in H^{1/2,1/4}(\Sigma)$, where the Hilbert spaces $H^{1/2,1/4}(\Sigma)$ and $W(Q)$ are given by \eqref{hrsa} and \eqref{w} respectively.
\end{proposition}

The proof of the above proposition can be obtained by proceeding as in \cite[Proposition 4]{araruna2019boundary} or \cite[Proposition 5]{liliana2020}.

\begin{remark}\label{rmqwell}$ $
	
\begin{enumerate}
	\item Using the Proposition \ref{well} and from the characterization of the optimal control given by \eqref{uop}, we have that 
	$$
	\hat{u}\textbf{1}_{\Gamma}\in H^{1/2,1/4}(\Sigma).
	$$
	\item The regularity of the follower in $H^{1/2,1/4}(\Sigma)$ instead of $L^2(\Sigma)$ is very important because, with this regularity, we can use the Carleman inequality with non homogeneous boundary condition belonging to $H^{1/2,1/4}(\Sigma)$. 
	
	\item If in the system \eqref{zp}, we consider $p(T,\cdot)=p^T$ with $p^T\in H^1_0(\Omega)$, then we can obtain the same well-posedness result as in the Proposition \ref{well}.
\end{enumerate}
\end{remark}

Now, we want to study the null controllability of the system \eqref{zp}. Arguing as in Section \ref{obsernull}, solving the null controllability problem of the system \eqref{zp} require the proof of a suitable observability inequality for the following adjoint system
\begin{equation}\label{adjoin}
	\left\{
	\begin{array}{ccl}
		\dis -\rho_{t}-\Delta \rho+\int_\Omega K(t,x,\theta)\rho(t,\theta)\ d\theta =\dis \psi\chi_{\O_d}& \mbox{in}& Q,\\
		\dis \psi_{t}-\Delta \psi+\int_\Omega K(t,x,\theta)\psi(t,\theta)\ d\theta =0& \mbox{in}& Q,\\
		\dis \rho=0,\ \ \ \psi=\frac{1}{\mu}\frac{\partial \rho}{\partial\nu}\textbf{1}_{\Gamma}& \mbox{on}& \Sigma, \\
		\dis  \rho(T,\cdot)=\rho^T,\ \ \ \psi(0,\cdot)=0 &\mbox{in}&\Omega,
	\end{array}
	\right.
\end{equation}	
where $\rho^T\in L^2(\Omega)$. 


The observability estimate associated to \eqref{adjoin} is given in the following proposition.
\begin{proposition} \label{observq}
	Assume  that $\O_d\cap\omega\neq \emptyset$ and $\mu$ is large enough. Then, there exist a positive constant $C$ and a weight function $\varpi_2=\varpi_2(t)$ blowing up at $t=T$ such that for any $\rho^T\in L^2(\Omega)$, the following inequality holds
	\begin{equation}\label{inequality}
		\|\rho(0,\cdot)\|^2_{L^2(\Omega)}+\int_{Q}\frac{1}{\varpi_2^2}|\psi|^2\,\dT
		\leq C\int_{0}^{T}\int_{\omega}|\rho|^2\,\dq,
	\end{equation}
	where the adjoint variables $(\rho, \psi)$ are the solution of \eqref{adjoin}.
\end{proposition}
Thanks to Proposition \ref{observq} and proceeding as in the Proposition \ref{linear}, we can prove that for each $z^0\in L^2(\Omega)$, there exists a control $\hat{g}\in L^2((0,T)\times\omega)$ such that the associated solution $(z, p)$ to \eqref{zp} satisfies \eqref{qnull}.

\begin{remark}
	Due to the density of $H^1_0(\Omega)$ in $L^2(\Omega)$, it can be assumed that $\rho^T\in H^1_0(\Omega)$. Since systems \eqref{zp} and \eqref{adjoin} have similar structures, according to Remark \ref{rmqwell}, the system \eqref{adjoin} admits a unique solution $(\rho,\psi)\in H^{2,1}(Q)\times W(Q)$.
\end{remark}

Now, we want to prove the inequality \eqref{inequality}. Thanks to hypothesis $\O_d\cap\omega\neq \emptyset$, then, there exists a nonempty set $\omega^\prime$ such that $\omega^\prime\subset\subset\O_d\cap\omega$ and $\eta^0\in \mathcal{C}^2(\overline{\Omega})$ a function satisfying \eqref{c3}.\\
For $(t,x)\in Q$, we consider the following functions
\begin{equation}\label{etaetoile}
	\begin{array}{rll}
	\dis \eta^*(t)=\dis  \max_{x\in \overline{\Omega}}\eta(t,x),\ \ \ \varphi^*(t)=\dis  \min_{x\in \overline{\Omega}}\varphi(t,x),	
	\end{array}
\end{equation}
where $\eta$ and $\varphi$ are the weight functions defined by \eqref{eta}.

We need the following notation in the sequel
\begin{equation}\label{n}
	\mathcal{N}(z)=s^{-1}\int_{Q}e^{-2s\eta}\varphi^{-1}|\nabla z|^2dxdt+s\lambda^2\int_{Q}e^{-2s\eta}\varphi| z|^2dxdt.
\end{equation}
For $F\in L^2(Q)$, $z_0\in L^2(\Omega)$ and $h\in H^{1/2,1/4}(\Sigma)$, we consider the following system
\begin{equation}\label{eqboundary}
	\left\{
	\begin{array}{rllll}
		\dis \frac{\partial{z}}{\partial{t}}-\Delta z&=&F& \mbox{in}& Q,\\
		\dis z&=&h& \mbox{on}& \Sigma,\\
		\dis z(0,\cdot)&=&z_0& \mbox{in}& \Omega.
	\end{array}
	\right.
\end{equation}
The observability inequality \eqref{inequality} is established via a global Carleman estimate given in \cite{ImanuvilovPuelYamamoto} for the solution of system \eqref{eqboundary}. More precisely, we have the following result.

\begin{lemma} \label{prog}
	Let $F\in L^2(Q)$, $z_0\in L^2(\Omega)$ and $h\in H^{1/2,1/4}(\Sigma)$. Then, there exists $s_0\geq1$ and $\lambda_0\geq1$ and there exists a constant $C_0>0$ independent of $s$ and $\lambda$ such that for all $s\geq s_0$ and $\lambda\geq \lambda_0 $, and for any $z\in W(Q)$ solution of \eqref{eqboundary}, we have
	\begin{equation}\label{carlg}
		\begin{array}{llll}
			\dis \mathcal{N}(z)&\leq& 
			\dis C_0\left(s^{-\frac{1}{2}}\|\varphi^{-\frac{1}{4}}he^{-s\eta}\|^2_{H^{1/2,1/4}(\Sigma)}+
			s^{-2}\lambda^{-2}\int_{Q} e^{-2s\eta}\varphi^{-2}|F|^2\ dxdt\right.\\
			&&\dis \left.+s\lambda^2\int_{0}^{T}\int_{\omega^\prime}e^{-2s\eta}\varphi |z|^2\ dxdt\right).
		\end{array}	
	\end{equation}
\end{lemma}

Thanks to Lemma \ref{prop4} and Lemma \ref{prog}, we can prove the Carleman inequality associated to adjoint system \eqref{adjoin}. We have the following result.
\begin{proposition} \label{carlemang} 
	Assume  that $\omega\cap\O_d\neq \emptyset$ and $\mu$ is large enough. Then, there exist a  positive constant $C=C(\Omega, \omega,T)>0$ such that the solution $(\rho,\psi)$ to \eqref{adjoin} satisfies
	\begin{equation}\label{robg}
		\begin{array}{lll}
			\dis \mathcal{M}(\rho)+\mathcal{N}(\psi)\leq
			\dis Cs^5\lambda^6\int_0^T\int_{\omega}e^{-2s\eta}\varphi^5|\rho|^2\dT,
		\end{array}
	\end{equation}		
	for every $s>0$ large enough.
\end{proposition}

\begin{proof}
We apply the estimate \eqref{car} to the first equation of \eqref{adjoin}, then we apply the estimate \eqref{carlg} to the second equation for \eqref{adjoin} and add them up, we obtain
\begin{equation*}
	\begin{array}{lll}
		\mathcal{M}(\rho)+\mathcal{N}(\psi)\\
			\leq\dis  C \left(\int_0^T\int_{\omega^\prime}e^{-2s\eta}\left(s^3\lambda^4\varphi^3|\rho|^2+s\lambda^2\varphi|\psi|^2\right)\dT+\int_{Q}e^{-2s\eta}|\psi|^2\dT\right)\\
		+\dis C\left( \int_{Q}e^{-2s\eta}\left|\int_\Omega K(t,x,\theta)\rho(t,\theta)\ d\theta\right|^2\dT+\int_{Q}e^{-2s\eta}\left|\int_\Omega K(t,x,\theta)\psi(t,\theta)\ d\theta\right|^2\dT\right)\\
		+\dis Cs^{-\frac{1}{2}}\left\|\varphi^{-\frac{1}{4}}e^{-s\eta}\frac{1}{\mu}\frac{\partial \rho}{\partial\nu}\textbf{1}_{\Gamma}\right\|^2_{H^{1/2,1/4}(\Sigma)},
	\end{array}
\end{equation*}
for any $s\geq \hat{s}$, $\lambda\geq \hat{\lambda}$, where $\hat{s}=\max(s_0,s_1)$, $\hat{\lambda}=\max(\lambda_0,\lambda_1)$ and the constant $C=C(C_0,\Omega,\omega)>0$.

Now, we fix the parameter $\lambda=\hat{\lambda}$ sufficiently large.  From \eqref{esta} and \eqref{estb}, we can absorb some terms in the right hand side of the above estimate and we obtain
\begin{equation}\label{termg}
	\begin{array}{lll}
		\mathcal{M}(\rho)+\mathcal{N}(\psi)\\
			\leq\dis  C\left( \int_0^T\int_{\omega^\prime}e^{-2s\eta}\left(s^3\lambda^4\varphi^3|\rho|^2+s\lambda^2\varphi|\psi|^2\right)\dT+\dis s^{-\frac{1}{2}}\left\|\varphi^{-\frac{1}{4}}e^{-s\eta}\frac{1}{\mu}\frac{\partial \rho}{\partial\nu}\textbf{1}_{\Gamma}\right\|^2_{H^{1/2,1/4}(\Sigma)}\right),
	\end{array}
\end{equation}
where $C=C(C_0,\Omega,\omega)>0$ and $C_0$ is the constant defined to Lemma \ref{prog}.\\
Now, we want to estimate the last term in \eqref{termg}. Firstly, note that the map $\dis \rho\longmapsto \frac{\partial \rho}{\partial\nu}\textbf{1}_{\Gamma}$ is a continuous and linear from $H^{2,1}(Q)$ to $ H^{1/2,1/4}(\Sigma)$. Secondly, due to the definition of weight functions given in \eqref{eta} and \eqref{etaetoile}, we have that $\eta$ and $\varphi$ are equals to $\eta^*$ and $\varphi^*$ on the boundary. Then, we can write
\begin{equation}\label{boun}
	\begin{array}{lll}
		\dis s^{-\frac{1}{2}}\left\|\varphi^{-\frac{1}{4}}e^{-s\eta}\frac{1}{\mu}\frac{\partial \rho}{\partial\nu}\textbf{1}_{\Gamma}\right\|^2_{H^{1/2,1/4}(\Sigma)}	\leq\dis  \frac{1}{\mu}s^{-\frac{1}{2}}\left\|(\varphi^*)^{-\frac{1}{4}}e^{-s\eta^*}\rho\right\|^2_{H^{2,1}(Q)}.
	\end{array}
\end{equation}
The strategy now is to estimate the term in the right hand side of \eqref{boun}. We set $\widehat{\rho}=\sigma\rho$ with $\sigma\in \mathcal{C}^\infty([0,T])$ defined by 
\begin{equation}\label{sigma}
\sigma= e^{-s\eta^*}(s\varphi^*)^{-1/4}.	
\end{equation}
 Observing that $\sigma(T)=0$ and using the definition of weight functions $\eta^*$ and $\varphi^*$, we obtain  for $s$ large enough
\begin{equation}\label{sigmat}
	|\sigma_t|\leq Ce^{-s\eta^*}(s\varphi^*).
\end{equation} 

Using the first equation of \eqref{adjoin}, $\widehat{\rho}$ satisfies the following system
\begin{equation}\label{can}
	\left\{
	\begin{array}{rllll}
		\dis -\widehat{\rho}_t-\Delta \widehat{\rho}+\int_\Omega K(t,x,\theta)\widehat{\rho}(t,\theta)\ d\theta &=&\dis\sigma\psi\chi_{\O_d}-\sigma_t\rho& \mbox{in}& Q,\\
		\dis \widehat{\rho}&=&0& \mbox{on}& \Sigma,\\
		\dis \widehat{\rho}(T,\cdot)&=&0& \mbox{in}& \Omega.
	\end{array}
	\right.
\end{equation}
Using the Proposition \ref{well}, we deduce that the system \eqref{can} admits a unique solution $\widehat{\rho}\in H^{2,1}(Q)$. Furthermore, the classical energy estimate associated to $\widehat{\rho}$ is given by	
\begin{equation}\label{estsigma}
	\begin{array}{llllll}
		\dis \|\widehat{\rho}\|^2_{H^{2,1}(Q)}\leq 
		C(\|K\|_{\infty}, T)\left(\|\sigma\psi\|^2_{L^2(Q)}+\|\sigma_t\rho\|^2_{L^2((Q)}\right).
	\end{array}
\end{equation} 		
Combining \eqref{sigma} and \eqref{sigmat} together with estimate \eqref{estsigma}, we have 	
\begin{equation}\label{estsigmabis} 
	\begin{array}{llllll}
		\dis s^{-\frac{1}{2}}\left\|(\varphi^*)^{-\frac{1}{4}}e^{-s\eta^*}\rho\right\|^2_{H^{2,1}(Q)}\\
		\dis\leq 
		C(\|K\|_{\infty}, T)\left(\int_Q e^{-2s\eta^*}(s\varphi^*)^{-1/2} |\psi|^2\ dxdt+\int_Q e^{-2s\eta^*}(s\varphi^*)^{2} |\rho|^2\ dxdt\right).
	\end{array}
\end{equation} 			
Using the definition of the weight functions $\eta^*$ and $\varphi^*$, then combining \eqref{estsigmabis} and \eqref{boun}, we get
\begin{equation}\label{estsi}
	\begin{array}{llllll}
		\dis s^{-\frac{1}{2}}\left\|\varphi^{-\frac{1}{4}}e^{-s\eta}\frac{1}{\mu}\frac{\partial \rho}{\partial\nu}\textbf{1}_{\Gamma}\right\|^2_{H^{1/2,1/4}(\Sigma)}\\
		\dis\leq 
		C(\|K\|_{\infty},\mu, T)\left(\int_Q e^{-2s\eta}s\varphi |\psi|^2\ dxdt+\int_Q e^{-2s\eta}(s\varphi)^{2} |\rho|^2\ dxdt\right).
	\end{array}
\end{equation} 	
Substituting \eqref{estsi} in \eqref{termg} and taking $s$ sufficiently large, we have 
\begin{equation}\label{termgbis}
	\begin{array}{lll}
	\mathcal{M}(\rho)+\mathcal{N}(\psi)
		\leq\dis  C(C_0,\Omega,\omega) \int_0^T\int_{\omega^\prime}e^{-2s\eta}\left(s^3\lambda^4\varphi^3|\rho|^2+s\lambda^2\varphi|\psi|^2\right)\dT.
	\end{array}
\end{equation}

	The last step is to eliminate the local term corresponding to $\psi$  on the right hand side of the estimate \eqref{termgbis}. We proceed as in the second step of the proof of Proposition \ref{propcarl}. We consider the open set $\omega_0$ and the function $\xi$ defined by \eqref{owogene1}.

We define $u=s\lambda^2\varphi e^{-2s\eta}$. Using the definition of $\varphi$ and $\eta$ given in \eqref{eta}, we get $u(T)=u(0)=0$. Due to the estimates \eqref{gradien}, we have the following:
\begin{equation}\label{conbn}
	\begin{array}{rll}
		\dis |u\xi|\leq s\lambda^2\varphi e^{2s\eta}\xi,\ \ \ \ \ \
		\dis \left|(u\xi)_t\right|\leq C(T)s^2\lambda^2\varphi^3e^{2s\eta}\xi,\\
		\\ 
		\dis |\nabla(u\xi)|\leq Cs^3\lambda^3\varphi^2e^{2s\eta}\xi,\ \ \ \ \ \
		\dis |\Delta(u\xi)|\leq Cs^3\lambda^4\varphi^3e^{2s\eta}\xi,
	\end{array}
\end{equation}
where $C$ is a positive constant.\\ 
Multiplying the first equation of \eqref{adjoin}  by $u\xi\psi$, integrating by parts over $Q$ and using the Fubini's theorem, we obtain
\begin{equation}\label{beaubn}
	I_1+I_2+I_3=\int_{Q}u\xi|\psi|^2\chi_{\O_d}\ \dT,
\end{equation}
where
\begin{equation*}
	I_1=\dis \int_{Q}\rho\psi(u\xi)_t\ \dT
	\leq\dis \frac{\gamma_1}{2}\int_{Q}  s\lambda^2\varphi e^{-2s\eta}|\psi|^2\dT+\frac{C(T)}{\gamma_1}\int_{0}^{T}\int_{\omega_0} s^3\lambda^2\varphi^5e^{-2s\eta}|\rho|^2\dT,
\end{equation*}
\begin{equation*}
	I_2=\dis-\int_Q\rho\psi\Delta(u\xi)\ \dT
	\leq \frac{\gamma_2}{2}\int_{Q} s\lambda^2\varphi e^{-2s\eta}|\psi|^2\dT+\frac{C}{\gamma_2}\int_{0}^{T}\int_{\omega_0} s^5\lambda^6\varphi^5e^{-2s\eta}|\rho|^2\dT,
\end{equation*}
\begin{equation*}
	I_3= \dis -2\int_{Q}\rho\nabla\psi.\nabla(u\xi)\ \dT
	\leq \frac{\gamma_3}{2}\int_{Q}s^{-1}\varphi^{-1} e^{-2s\eta}|\nabla\psi|^2\dT+\frac{C}{\gamma_3}
	\int_{0}^{T}\int_{\omega_0} s^{5/2}\lambda^6\varphi^e{5/2}^{-2s\eta}|\rho|^2\dT,	
\end{equation*}
for any $\gamma_i>0,\ i=1,2,3$.

By making the right choice of $\gamma_i,\ i=1,2,3$, we obtain that
\begin{equation}\label{owo4bn}
	\begin{array}{rll}
		\dis \int_{0}^{T}\int_{\omega^\prime} s\lambda^2\varphi e^{-2s\eta}|\psi|^2\dT\dis \leq \dis \frac{1}{2}\mathcal{N}(\psi)+\dis C(T)\int_0^T\int_{\omega_0}s^5\lambda^6\varphi^5e^{-2s\eta}|\rho|^2 \ \dT.
	\end{array}
\end{equation}
We replace \eqref{owo4bn} in \eqref{termgbis} and since $\omega_0\subset \omega$, we deduce \eqref{robg}.   

\end{proof}

Now, we are ready to prove the observability inequality \eqref{inequality}.
\paragraph{}
\textbf{Proof of Proposition \ref{observq}.}
We proceed as in the proof of Proposition \ref{pro}. The idea is to improve \eqref{robg} in the sense that the weight functions do not degenerate at $t=0$. To this end, we consider functions $\alpha,\ \zeta,\ \alpha^*$ and $\zeta^*$ defined in \eqref{alpha}.

 We fix the parameter $s$ to a sufficiently large fixed value. Firstly, by construction, $\eta=\alpha$ and $\varphi=\zeta$ in $[T/2,T]\times \Omega$, hence, 
\begin{equation}\label{pou5bn}
	\begin{array}{llll}
		\dis \int_{T/2}^{T}\int_{\Omega}e^{-2s\alpha}\zeta^3|\rho|^2\ dxdt+\int_{T/2}^{T}\int_{\Omega}e^{-2s\alpha}\zeta|\psi|^2\ dxdt\\
		=\dis \int_{T/2}^{T}\int_{\Omega}e^{-2s\eta}\varphi^3|\rho|^2\ dxdt+\int_{T/2}^{T}\int_{\Omega}e^{-2s\eta}\varphi|\psi|^2\ dxdt\\
		\leq\dis C(\Omega, \omega,T)\int_{0}^{T}\int_{\omega}e^{-2s\eta}\varphi^5|\rho|^2\,\dq,
	\end{array}
\end{equation}
where we have used the Carleman estimate \eqref{robg}.

On the other hand, proceeding as in the second step of the proof of Proposition \ref{pro}, we obtain
$$
\begin{array}{llll}
	\dis \|\rho(0,\cdot)\|^2_{L^2(\Omega)}+\int_{0}^{T/2}\int_{\Omega}|\rho|^2\ dxdt\\
	\dis \leq C(\|K\|_{\infty},T)
	\left(\int_0^{3T/4}\int_\Omega |\psi|^2\ \dq
	+\int_{T/2}^{3T/4}\int_\Omega |\rho|^2\ \dq\right).
\end{array}
$$
Adding the term $\dis \int_{0}^{T/2}\int_\Omega |\psi|^2\ \dq$ on both sides of above estimate, we arrive to
\begin{equation}\label{pou1bn}
\begin{array}{llll}
	\dis \|\rho(0,\cdot)\|^2_{L^2(\Omega)}+\int_{0}^{T/2}\int_{\Omega}|\rho|^2\ dxdt+\int_{0}^{T/2}\int_\Omega |\psi|^2\ \dq\\
	\dis \leq C(\|K\|_{\infty},T)
	\left(\int_{T/2}^{3T/4}\int_\Omega (|\psi|^2+|\rho|^2)
	\ \dq+\int_{0}^{T/2}\int_\Omega |\psi|^2\ \dq\right).
\end{array}	
\end{equation}
In order to eliminate the last term in the right hand side of \eqref{pou1bn}, we use the classical energy estimates for the second equation of system \eqref{adjoin} in the domain $(0,T/2)\times \Omega$ and we obtain:
\begin{equation}\label{poubn}
	\begin{array}{llll}
		\dis \int_{0}^{T/2}\int_{\Omega}|\psi|^2\ dxdt
		\dis &\leq&\dis  \frac{1}{\mu^2}C(\|K\|_{\infty}, T) \left\|\frac{\partial \rho}{\partial\nu}\textbf{1}_{\Gamma}\right\|^2_{H^{1/2,1/4}((0,T/2)\times \partial\Omega)},\\
		&\leq& \dis  \frac{1}{\mu^2}C(\|K\|_{\infty}, T)\left\|\rho\right\|^2_{H^{2,1}((0,T/2)\times \Omega)},\\
		&\leq&\dis  \frac{1}{\mu^2}C(\|K\|_{\infty}, T)\left\|(\zeta^*)^{-\frac{1}{4}}e^{-s\alpha^*}\rho\right\|^2_{H^{2,1}((0,T/2)\times \Omega)}\\
		&\leq&\dis  \frac{1}{\mu^2}C(\|K\|_{\infty}, T)\left\|(\zeta^*)^{-\frac{1}{4}}e^{-s\alpha^*}\rho\right\|^2_{H^{2,1}(Q)}
	\end{array}
\end{equation}
because the map $\dis \rho\longmapsto \frac{\partial \rho}{\partial\nu}\textbf{1}_{\Gamma}$ is continuous and linear from $H^{2,1}(Q)$ to $H^{1/2,1/4}(\Sigma)$ and the functions $\alpha^*$ and $\zeta^*$ are constants in both variables in $(0,T/2)\times \Omega$.

Replacing \eqref{poubn} in \eqref{pou1bn}, and using the fact that the weight functions $\alpha$ and $\zeta$ defined in \eqref{alpha} are bounded in $[0,3T/4]\times \Omega$, we obtain
\begin{equation}\label{pou3bn}
	\begin{array}{llll}
		\dis \|\rho(0,\cdot)\|^2_{L^2(\Omega)}+\int_{0}^{T/2}\int_{\Omega}e^{-2s\alpha}\zeta^3|\rho|^2\ dxdt+\int_{0}^{T/2}\int_\Omega e^{-2s\alpha}\zeta|\psi|^2\ \dq\\
		\dis \leq C(\|K\|_{\infty},T)
		\int_{T/2}^{3T/4}\int_\Omega e^{-2s\alpha}(\zeta|\psi|^2+\zeta^3|\rho|^2)
		\ \dq+ \frac{1}{\mu^2}C(\|K\|_{\infty}, T)\left\|(\zeta^*)^{-\frac{1}{4}}e^{-s\alpha^*}\rho\right\|^2_{H^{2,1}(Q)}.
	\end{array}	
\end{equation}
	
Adding estimates \eqref{pou5bn} and \eqref{pou3bn} and using definitions of $\alpha^*$ and $\zeta^*$, we have	
\begin{equation}\label{canbn}
	\begin{array}{llll}
		\dis \|\rho(0,\cdot)\|^2_{L^2(\Omega)}+\int_{0}^{T}\int_{\Omega}e^{-2s\alpha^*}(\zeta^*)^3|\rho|^2\ dxdt+\int_{0}^{T}\int_\Omega e^{-2s\alpha^*}\zeta^*|\psi|^2\ \dq\\
		\dis \leq\dis C(\Omega, \omega,T)\int_{0}^{T}\int_{\omega}e^{-2s\eta}\varphi^5|\rho|^2\,\dq+ \frac{1}{\mu^2}C(\|K\|_{\infty}, T)\left\|(\zeta^*)^{-\frac{1}{4}}e^{-s\alpha^*}\rho\right\|^2_{H^{2,1}(Q)}.
	\end{array}	
\end{equation}	
To conclude, we want to estimate the last norm in the previous inequality. To this end, we use system \eqref{can} and estimate \eqref{estsigma} with $\dis \sigma=e^{-s\alpha^*}(\zeta^*)^{-1/4}$ and we obtain 

\begin{equation}\label{estsigmabn} 
	\begin{array}{llllll}
		\dis \left\|e^{-s\alpha^*}\zeta^*)^{-\frac{1}{4}}\rho\right\|^2_{H^{2,1}(Q)}\\
		\dis\leq 
		C(\|K\|_{\infty}, T)\left(\int_Q e^{-2s\alpha^*}(\zeta^*)^{-1/2} |\psi|^2\ dxdt+\int_Q e^{-2s\alpha^*}(\zeta^*)^{2} |\rho|^2\ dxdt\right).
	\end{array}
\end{equation}
Replacing \eqref{estsigmabn} in \eqref{canbn}, using the fact that $(\zeta^*)^{-1},\ (\zeta^*)^{-3/2}\in L^\infty(Q)$ and taking $\mu$ large enough, we obtain 
\begin{equation}\label{canbnd}
	\begin{array}{llll}
		\dis \|\rho(0,\cdot)\|^2_{L^2(\Omega)}+\int_{0}^{T}\int_{\Omega}e^{-2s\alpha^*}(\zeta^*)^3|\rho|^2\ dxdt+\int_{0}^{T}\int_\Omega e^{-2s\alpha^*}\zeta^*|\psi|^2\ \dq\\
		\dis \leq\dis C(\Omega, \omega,T)\int_{0}^{T}\int_{\omega}e^{-2s\eta}\varphi^5|\rho|^2\,\dq.
	\end{array}	
\end{equation}	

We obtain the observability inequality \eqref{inequality} due to the fact that $\dis e^{-2s\eta}\varphi^5\in L^\infty(Q)$ and taking 
\begin{equation}\label{varpibn}
	\varpi_2(t)=e^{s\alpha^*}(\zeta^*)^{-1/2}.
\end{equation}

\subsection{The semilinear case}

In the same spirit of Section \ref{nonlin}, we consider a more general version of system  \eqref{eqsemi} given by
\begin{equation}\label{eqsemibn}
	\left\{
	\begin{array}{rllll}
		\dis y_{t}-\Delta y+\int_\Omega K(t,x,\theta)G(y(t,\theta))\ d\theta &=&g\chi_{\omega}+v\chi_{\O}& \mbox{in}& Q,\\
		\dis y&=&u\textbf{1}_{\Gamma}& \mbox{on}& \Sigma, \\
		\dis  y(0,\cdot)&=&y^0 &\mbox{in}&\Omega,
	\end{array}
	\right.
\end{equation}
where $G$ is a globally Lipschitz continuous function. 

As in the Section \ref{nonlin}, in the semilinear case, we lose the convexity of the functional $\widetilde{J}$ defined by \eqref{j}. For this, we state the following result whose the proof can be obtained by proceeding in the same way as in the Proposition \ref{convexity}.

\begin{proposition}\label{convexitybn}
	
	Let us assume that $q_{d}\in L^\infty((0,T); \O_d)$, $g\in L^2((0,T)\times\omega)$ are fixed and $\mu$ is sufficiently large. Suppose that  $q^0\in L^2(\Omega)$, $N\leq 6$, $G\in\mathcal{C}^2(\R)$ and there exists a constant $L>0$ such that $\dis |G^\prime(s)|+|G^{\prime\prime}(s)|\leq L,\ \forall s\in \R$.
	Then, if $\hat{u}$ satisfies \eqref{hatu}, there exists a constant $C>0$ independent of $\mu$ such that 
	\begin{equation}
		\dis D^2\widetilde{J} (g;\hat{u})\cdot(w^1,w^1)\geq C \|w^1\|^2_{L^2((0,T)\times \Gamma)},\ \forall w^1\in L^2((0,T)\times \Gamma),\ w^1\neq 0.
	\end{equation}
	In particular, the functional $\widetilde{J}$ is strictly convex in $ \hat{u}$.	
\end{proposition}

Proceeding as in the Proposition \ref{quasi}, the optimal control $\hat{u}$ is given by
\begin{equation}\label{uopbn}
	\hat{u}=\frac{1}{\mu}\frac{\partial p}{\partial\nu}\ \ \mbox{on}\ \ \ (0,T)\times \Gamma,	
\end{equation}
where  $(q,p)$ is the solution of optimality system
\begin{equation}\label{zpbn}
	\left\{
	\begin{array}{ccl}
		\dis q_{t}-\Delta q+\int_\Omega K(t,x,\theta)G(q(t,\theta))\ d\theta =\dis g\chi_{\omega}& \mbox{in}& Q,\\
		\dis -p_{t}-\Delta p+\int_\Omega K(t,x,\theta)G^\prime(q(t,\theta))p(t,\theta)\ d\theta =(q-q_d)\chi_{\O_d}& \mbox{in}& Q,\\
		\dis q=\frac{1}{\mu}\frac{\partial p}{\partial\nu}\textbf{1}_{\Gamma},\ \ \ p=0& \mbox{on}& \Sigma, \\
		\dis  q(0,\cdot)=q^0,\ \ \ p(T,\cdot)=0 &\mbox{in}&\Omega.
	\end{array}
	\right.
\end{equation}		

The following result holds.
\begin{theorem}\label{theosemibn}
	Under the assumptions of Proposition \ref{convexitybn}, there exists a positive  weight function $\varpi_2=\varpi_2(t)$ blowing up at $t=T$ such that for any $q_{d}\in L^2((0,T)\times\O_d)$ satisfying \eqref{noubis}, there exist a control $\hat{g}\in L^2((0,T)\times \omega)$ and a unique optimal control $\hat{u}\in H^{1/2,1/4}((0,T)\times\Gamma)$ such that the corresponding solution to \eqref{zpbn} satisfies \eqref{obj}. 
\end{theorem}

\begin{proof}
	In order to prove Theorem \ref{theosemibn}, we can use the ideas of Theorem \ref{theosemi} given in the Section \ref{nonlin}. Let us take $z=q-\bar{q}$, where $q$ is the solution of \eqref{eqbis} and $\bar{q}$ is the solution of \eqref{eqbisbar}.  Let us rewrite \eqref{zpbn} of the form
	\begin{equation}\label{zpsemibn}
		\left\{
		\begin{array}{ccl}
			\dis z_{t}-\Delta z+\int_\Omega K(t,x,\theta)a(z(t,\theta))z(t,\theta)\ d\theta =\dis g\chi_{\omega}& \mbox{in}& Q,\\
			\dis -p_{t}-\Delta p+\int_\Omega K(t,x,\theta)b(z(t,\theta))p(t,\theta)\ d\theta =(z-z_d)\chi_{\O_d}& \mbox{in}& Q,\\
			\dis z=\frac{1}{\mu}\frac{\partial p}{\partial\nu}\textbf{1}_{\Gamma},\ \ \ p=0& \mbox{on}& \Sigma, \\
			\dis  z(0,\cdot)=z^0,\ \ \ p(T,\cdot)=0 &\mbox{in}&\Omega,
		\end{array}
		\right.
	\end{equation}
	where $z^0:=q^0-\bar{q}^0$, $z_d:=q_d-\bar{q}$ and 
	$$
	a(z)=\int_0^1G^\prime(\bar{y}+s z)\ ds,\ \ b(z)=G^\prime(z+\bar{y}).
	$$
	The condition \eqref{obj} is equivalent to a null controllability property for $z$, that is
	\begin{equation}\label{znullsemibn}
		z(T,x)=0\ \mbox{in}\ \Omega.
	\end{equation}	
	
	Let us consider for each $w\in L^2(Q)$ the following linear system associated to \eqref{zpsemibn}
	\begin{equation}\label{zplinearbn}
		\left\{
		\begin{array}{ccl}
			\dis z_{t}-\Delta z+a(w)\int_\Omega K(t,x,\theta)z(t,\theta)\ d\theta =\dis g\chi_{\omega}& \mbox{in}& Q,\\
			\dis -p_{t}-\Delta p+b(w)\int_\Omega K(t,x,\theta)p(t,\theta)\ d\theta =(z-z_d)\chi_{\O_d}& \mbox{in}& Q,\\
			\dis z=\frac{1}{\mu}\frac{\partial p}{\partial\nu}\textbf{1}_{\Gamma},\ \ \ p=0& \mbox{on}& \Sigma, \\
			\dis  z(0,\cdot)=z^0,\ \ \ p(T,\cdot)=0 &\mbox{in}&\Omega,
		\end{array}
		\right.
	\end{equation}
	
	Thanks to the hypothesis on $G$ given by the Proposition \ref{convexitybn}, there exists a constant $M>0$ such that
	$$
	\|a(w)\|_{L^\infty(Q)},\ \|b(w)\|_{L^\infty(Q)}\leq M,\ \forall w\in L^2(Q).
	$$
	
	Thanks to Proposition \ref{observq}, we can prove the following observability inequality
	\begin{equation}\label{obsersemibn}
		\|\rho(0,\cdot)\|^2_{L^2(\Omega)}+\int_{Q}\frac{1}{\varpi_2^2 }|\psi|^2\,\dT
		\leq C\int_{0}^{T}\int_{\omega}|\rho|^2\,\dq,
	\end{equation}
	where $C=C(\Omega,\omega, \|K\|_{\infty},\|a(w)\|_{L^\infty(Q)},\|b(w)\|_{L^\infty(Q)}, T)>0$, the function $\varpi_2$ is given by \eqref{varpibn} and $(\rho, \psi)$ is the solution to the following adjoint system associated to \eqref{zplinearbn}:
	\begin{equation}\label{}
		\left\{
		\begin{array}{ccl}
			\dis -\rho_t-\Delta \rho+a(w)\int_\Omega K(t,x,\theta)\rho(t,\theta)\ d\theta =\dis \psi\chi_{\O_d}& \mbox{in}& Q,\\
			\dis \psi_t-\Delta \psi+b(w)\int_\Omega K(t,x,\theta)\psi(t,\theta)\ d\theta =0& \mbox{in}& Q,\\
			\dis \rho=0,\ \ \ \psi=\frac{1}{\mu}\frac{\partial \rho}{\partial\nu}\textbf{1}_{\Gamma}& \mbox{on}& \Sigma, \\
			\dis  \rho(T,\cdot)=\rho^T,\ \ \ \psi(0,\cdot)=0 &\mbox{in}&\Omega.
		\end{array}
		\right.
	\end{equation}

	Thanks to estimate \eqref{obsersemibn} and arguing as in the end of the proof of Theorem \ref{theosemi}, we obtain the controllability result of the semilinear system \eqref{eqsemibn}.
\end{proof}


\end{document}